\documentclass[hidelinks,onefignum,onetabnum]{siamart220329}

\usepackage{subfigure}
\usepackage{graphicx}
\usepackage{amsmath}
\usepackage{booktabs}
\usepackage{multirow}
\newtheorem{assumption}{Assumption}
\newtheorem{example}{Example}

\usepackage{lipsum}
\usepackage{amsfonts}
\usepackage{graphicx}
\usepackage{epstopdf}
\usepackage{algorithmic}
\ifpdf
  \DeclareGraphicsExtensions{.eps,.pdf,.png,.jpg}
\else
  \DeclareGraphicsExtensions{.eps}
\fi

\newsiamremark{remark}{Remark}
\newsiamremark{hypothesis}{Hypothesis}
\crefname{hypothesis}{Hypothesis}{Hypotheses}
\newsiamthm{claim}{Claim}

\headers{HPGSRN for $\ell_q$-norm composite optimization}{YUQIA WU, SHAOHUA Pan, and XIAOQI Yang}

\title{A regularized Newton method for $\ell_q$-norm composite optimization problems \thanks{Submitted to the editors DATE.
		\funding{The second author's research was supported by the National Natural Science Foundation of China under project No.11971177. The third author's research was partially supported by Research Grants Council of Hong Kong (PolyU 15217520).}}}

\author{Yuqia Wu\thanks{Department of Applied Mathematics, The Hong Kong Polytechnic University, Hong Kong  
		(\email{yuqia.wu@connect.polyu.hk}).}
	\and Shaohua Pan\thanks{School of Mathematics,  South China University of Technology, Guangzhou, China 
		(\email{shhpan@scut.edu.cn}).}
	\and Xiaoqi Yang\thanks{Department of Applied Mathematics, The Hong Kong Polytechnic University, Hong Kong  
		(\email{mayangxq@polyu.edu.hk}).}}

\usepackage{amsopn}

\usepackage{amssymb}
\usepackage{mathrsfs}

\ifpdf
\hypersetup{
	pdftitle={HpgSRN}
}
\fi

\begin{document}
	\maketitle
	
	\begin{abstract}
		This paper is concerned with $\ell_q\,(0<q<1)$-norm regularized minimization problems with a twice continuously differentiable loss function. For this class of nonconvex and nonsmooth composite problems, many algorithms have been proposed to solve them and most of which are of the first-order type. In this work, we propose a hybrid of proximal gradient method and subspace regularized Newton method, named HpgSRN. The whole iterate sequence produced by HpgSRN is proved to have a finite length and converge to an $L$-type stationary point under a mild curve-ratio condition and the Kurdyka-{\L}ojasiewicz property of the cost function, which does linearly if further a Kurdyka-{\L}ojasiewicz property of exponent $1/2$ holds. Moreover, a superlinear convergence rate for the iterate sequence is also achieved under an additional local error bound condition. Our convergence results do not require the isolatedness and strict local minimality properties of the $L$-stationary point. Numerical comparisons with ZeroFPR, a hybrid of proximal gradient method and quasi-Newton method for the forward-backward envelope of the cost function, proposed in [A. Themelis, L. Stella, and P. Patrinos, {\em SIAM J. Optim., } 28(2018), pp. 2274-2303] for the $\ell_q$-norm regularized linear and logistic regressions on real data indicate that HpgSRN not only requires much less computing time but also yields comparable even better sparsities and objective function values.
	\end{abstract}
	
	

	\section{Introduction}\label{sec1}
	
	Let $A\in \mathbb{R}^{m\times n}$ be a data matrix, and let $f\!:\mathbb{R}^m\rightarrow \mathbb{R}$ be a twice continuously differentiable function with $c_{f}\!:=\inf_{z\in\mathbb{R}^m}f(z)>\!-\infty$.
	We consider the following $\ell_q\,(0\!<\!q\!<\!1)$-norm regularized composite optimization problem
	\begin{equation}\label{model}
		\min_{x\in\mathbb{R}^n} F(x):=f(Ax)+\lambda \|x\|_q^q,
	\end{equation}
	where $\lambda>0$ is the regularization parameter and $\|x\|_q\!:=\left(\sum_{i=1}^n |x_i|^q\right)^{1/q}$
	denotes the $\ell_q$ quasi-norm of $x$.
	When $f(\cdot)=\|\cdot - b\|^2$ for a vector $b\in \mathbb{R}^m$,
	problem \cref{model} reduces to the familiar $\ell_q$-norm regularized least squares problem.
	
	Problem \eqref{model} first appears in statistics as the bridge penalty regression \cite{frank93},
	and later appears in optimization as a special case of nonsmooth and nonconvex penalty problems
	studied by Luo et al. \cite{luo96} and Yang et al. \cite{huang03,yang01} for
	nonlinear optimization problems. In signal processing, Chartrand \cite{chartrand07} early
	showed that the $\ell_q\,(0\!<\!q\!<\!1)$-norm can substantially reduce the number
	of measurements required by $\ell_1$-norm for exact recovery of signals, which motivates
	the frequent use of the $\ell_q\,(0\!<\!q\!<\!1)$-norm in compressed sensing.
	Because for any given $x\in\mathbb{R}^n$, $\|x\|_q^q\to \|x\|_0$ as $q\downarrow0$,
	where $\|x\|_0$ denotes the zero-norm (cardinality) of $x$, problem \eqref{model} is often used as a nonconvex surrogate
	of the zero-norm regularized problem, and is found to
	have a wide spectrum of applications in signal and image processing, statistics,
	and machine learning (see, e.g., \cite{saab08,nikolova08,figueiredo07,xu10,Xu23}).
	Inspired by the wide applications of \eqref{model}, this work aims to propose
	a globally convergent subspace regularized Newton method, i.e., a hybrid of proximal gradient (PG) method and subspace regularized Newton method, for solving it.
	\subsection{Related works}\label{sec1.1}
	
	Due to the nonconvexity and non-Lipschitz continuity of the $\ell_q$-norm,
	problem \eqref{model} is a class of difficult nonconvex and nonsmooth optimization
	problems. In the past decade, many first-order methods have been developed for
	seeking its critical points. For some special $q$, say $q=1/2$ or $2/3$,
	since the proximal mapping of the $\ell_q$-norm has a closed-form solution (see Xu et al. \cite{Xu12,Xu23}), the PG method becomes a class of popular ones for solving \eqref{model} with such $q$. For a general $q\in(0,1)$, Hu et al. also proposed an exact PG method in \cite{Hu17} and an inexact PG method in \cite{Hu21} for \cref{model}. When assuming that the limit point is a local minimizer, a linear convergence rate was obtained in \cite{Hu17,Hu21,Xu12}. In addition, a class of PG methods with a nonmonotone line search strategy (called SpaRSA) was proposed by Wright et al. \cite{Wright09}. For problem \eqref{model} with a general $q\in(0,1)$, the reweighted $\ell_1$-minimization method \cite{Lu14, chen14, lai13} is another class of common first-order methods
	by solving a sequence of weighted $\ell_1$-norm regularized minimization problems. To overcome the non-Lipschitz difficulty of the $\ell_q$-norm, Chen et al. \cite{Chen10} proposed a class of smoothing method by constructing a smooth approximation of the $\ell_q$-norm and using the steepest descent method to solve the constructed
	smooth approximation problem.
	
	As is well known, first-order methods have many advantages such as cheap computational cost in each iterate, weak global convergence conditions and easy implementation,
	but their convergence rate is at most linear. Second-order methods normally have a local superlinear convergence rate, but it is not an easy task to achieve the global convergence of their whole iterate sequence.  For unconstrained nonconvex and smooth optimization, the global convergence analysis of Newton-type methods armed with line search is limited to the subsequential convergence of the iterate sequence (see \cite{Wright02} and references therein),  let alone for problem \cref{model}. As will be discussed later, Themelis et al. \cite{Themelis18,themelis21} recently provided the global convergence analysis of Newton-type methods armed with line search for problem \cref{model} under the Kurdyka-{\L}ojasiewicz (KL) framework, but they assume that the second-order directions are controlled by the residual, 
	and now it is unclear what condition can ensure it to hold if Newton directions are adopted. In addition,
	for unconstrained smooth optimization problems, to the best of our knowledge, the current weakest condition for a second-order method to have a local superlinear convergence rate is a local error bound condition at local minima; see \cite{Ueda10,li04,fischer02}. 
	Therefore, it is natural to ask whether it is possible to design a globally convergent Newton-type method for problem \cref{model} with a local superlinear convergence rate under a similar local error bound condition.
	
	In recent years, many researchers are interested in using second-order methods to solve the following general nonconvex and nonsmooth composite problem
	\begin{equation}\label{composite}
		\min_{x\in\mathbb{R}^n} \Phi(x):= \phi(x)+ h(x),
	\end{equation}
	where $h\!:\mathbb{R}^n\to(-\infty,\infty]$ is a closed proper function and
	$\phi$ is twice continuously differentiable on an open subset containing
	the effective domain of $h$.
	For problem \eqref{composite} with both $\phi$ and $h$ convex,
	there are some active investigations in this direction.
	Lee et al. \cite{Lee14} proposed an inexact proximal Newton-type method
	and achieved the local quadratic
	convergence rate of the iterate sequence under the strong convexity of $\phi$. Yue et al. \cite{Yue19} proposed an inexact regularized proximal Newton method and
	established the local linear, superlinear and quadratic convergence rate
	of the iterate sequence (by the approximation degree to the Hessian matrix of $\phi$) under Luo-Tseng error bound. Mordukhovich et al. \cite{Mordu20} proposed a proximal Newton-type method and obtained the superlinear convergence rate of the iterate sequence under the metric $p\ (>1/2)$-subregularity of the subdifferential mapping $\partial\Phi$. We notice that the inexact proximal Newton-type method in \cite{Lee14} was  recently extended by Kanzow and Lechner \cite{Kanzow21} to solve problem \eqref{composite} with only a convex $h$, which essentially belongs to weakly convex optimization. Their global and local superlinear convergence results require the local strong convexity of $\Phi$ around any stationary point.
	
	By following a different line, the forward-backward envelope (FBE) of $\Phi$ has been
	extensively investigated for designing a hybrid of PG and second-order methods.
	For $h$ being convex with a cheap computable proximal mapping,
	Stella et al. \cite{stella17} combined  a PG method and a quasi-Newton method to minimize the FBE of $\Phi$ and proved the convergence of the whole sequence under the KL property of $\Phi$ and the superlinear convergence rate
	under the local strong convexity of the FBE of $\Phi$.
	For \eqref{composite} with $\phi$ being additionally convex and $h$ just having a cheap computable proximal mapping, Themelis et al. \cite{Themelis19} proposed a hybrid
	of PG and inexact Newton methods by using FBE of $\Phi$ (named FBTN) and proved that ${\rm dist}(x^k,\mathcal{X}^*)$ converges superlinearly to $0$ under an assumption without requiring the singleton of the solution set $\mathcal{X}^*$ of \eqref{composite}.
	In \cite{Themelis18} Themelis et al. used the FBE of $\Phi$ to develop a hybrid framework of PG and quasi-Newton methods (ZeroFPR), and achieved the global convergence of the iterate sequence by virtue of the KL property of the FBE, and its local superlinear rate
	under the Dennis-Mor\'{e} condition and the strong local minimum of the limit point.
	The convergence rate results in \cite{stella17,Themelis18} require the isolatedness of the limit point. Very recently, Ahookhosh et al. \cite{themelis21} utilized
	the Bregman FBE of $\Phi$ to develop a more general hybrid framework of PG and
	second-order methods, BELLA. They obtained the global convergence of the iterate sequence for the tame functions
	$\phi$ and $h$, and the local superlinear rate of the distance of the iterate sequence to the set of fixed points of the Bregman FBE by assuming that the second-order directions are the superlinear ones with order $1$ and KL property of exponent $\theta\in(0,1)$ of $\Phi$. Their work greatly improved the results of \cite{stella17,Themelis18} by removing
	the isolatedness restriction on local minima and established that the second-order directions are indeed the superlinear ones with order $1$ under the assumptions that the limit point is a strong local minimum (also implying the isolatedness) and a Dennis-Mor\'{e} condition holds. It is unclear what conditions are sufficient
	for second-order directions to be superlinear without the strong local minimum property.
	To achieve the global convergence, the search directions in \cite{stella17,Themelis18,themelis21} are required to be controlled by the residual (see \cref{dboundxkxbar}). 
	
	It is worth noting that when $\phi$ and $h$ in \eqref{composite} are convex,
	the local quadratic or superlinear convergence to a nonisolated optimal solution
	of second-order methods was obtained in \cite{Yue19,Mordu20}.
	To the best of our knowledge, the paper \cite{themelis21} is the first to achieve
	the superlinear convergence of the distance function of the iterate sequence
	without the isolateness assumption for solving problem \eqref{composite} with nonconvex and nonsmooth $h$.
	
	In addition, 
	for the case $h(x)=\lambda \|x\|_0$, Zhou et al. \cite{zhou21} developed a subspace Newton method by solving the stationary equations restricted in the subspace identified by the proximal mapping of $\lambda \|x\|_0$,
	and established
	the local quadratic convergence rate of the iterate sequence under
	the local strong convexity of $\phi$ around any stationary point.  Their subspace Newton method relies on the subspaces identified by a PG method. Recently, Bareilles et al. \cite{bareilles20} considered problem \eqref{composite} where $\phi$ is smooth and $h$ has a cheap computable proximal mapping, and proposed a hybrid  of PG and Newton methods under the framework of manifolds. Their algorithm alternates between a PG step and a Riemannian update on an identified manifold, and was proved to
	have a quadratic convergence rate under a positive definiteness assumption on the Riemannian Hessian of the objective function at limit points. 
	For the unified analysis on manifold identification of any PG methods, we refer the reader to the work \cite{sun19}.
	\subsection{Main contributions}\label{sec1.2}
	
	In this paper, we propose a hybrid of PG and subspace regularized Newton methods (HpgSRN)
	for problem \cref{model}. Though problem \eqref{model} is a special case of \eqref{composite}, our HpgSRN is quite different from ZeroFPR \cite{Themelis18} and BELLA \cite{themelis21}; see the discussions in \cref{remark-hybrid} (d). For convenience, in the sequel, we write
	\[
	\psi(x):=f(Ax)\ \mbox{ and }\ g(x):=\|x\|_q^q
	\ \mbox{ for } x\in\mathbb{R}^n.
	\]
	To describe the working flow of HpgSRN, for any given $S\subseteq\{1,2,\ldots,n\}$,
	define
	\begin{equation}\label{FS-fun}
		F_{S}(u)\!:=\psi_{S}(u)+\lambda g_S(u)\ \ {\rm with}\ \
		\psi_S(u)\!:=\!f(A_{S}u),g_S(u)\!:=\!\sum_{i\in S}|u_i|^q
		\ \ {\rm for}\ u\in\mathbb{R}^{|S|}.
	\end{equation}
	By \cref{property-FS}, for $S={\rm supp}(x)$,
	such $F_{S}$ is twice continuously differentiable at $x_{S}$.
	
	As mentioned before, a PG method needs a weak condition for global convergence
	and a very cheap computation cost in each iterate, but has only a linear convergence rate. Hence, we use a PG method to seek a good estimate in some neighborhood of
	a potential critical point, and enhance the convergence speed by using
	a regularized Newton method in the subspace associated to the support of the iterate
	generated by the PG method. Specifically, with the current $x^k$,
	the PG step yields $\overline{x}^k$ by computing
	\[
	\overline{x}^{k} \in \mathop{\arg\min}_{x\in \mathbb{R}^n}\Big\{\psi(x^k)+\langle \nabla\psi(x^k), x-x^k\rangle
	+ \frac{\overline{\mu}_k}{2} \|x-x^k\|^2+\lambda g(x)\Big\},
	\]
	where the step-size $\overline{\mu}_k$  depends on the (local) Lipschitz constant
	of $\nabla \psi$ near $x^k$.
	Write $S_k\!:={\rm supp}(x^k)$ and $S_k^{c}\!:=\{1,\ldots,n\}\backslash S_k$. 
	If 
	one of the conditions in \cref{if-else} is not satisfied, let $x^{k+1}\!:=\!\overline{x}^k$
	and return to the PG step; otherwise switch to a regularized Newton step where
	the Newton direction $d^k$ has the form $(d_{S_k}^k;0)$ with
	\begin{equation}\label{rnt-subp}
		d_{S_k}^k:=\mathop{\arg\min}_{d_{\!S_{k}} \in\mathbb{R}^{|S_{k}|}}
		\Big\{F_{S_k}(u^k)+\langle \nabla F_{S_k}(u^k), d_{S_k}\rangle +
		\frac{1}{2}\left\langle(\nabla^2 F_{S_k}(u^k)+\xi_k I)d_{S_k},d_{S_k}\right\rangle\Big\}
	\end{equation}
	for $u^k\!=\!x_{S_k}^k$ and a constant $\xi_k$ such that
	$\nabla^2 F_{S_k}(u^k)+\xi_k I$ is positive definite.
	It is easy to verify that $d^k_{S_k}$ is the unique solution of the system of linear equations
	\[
	\left(\nabla^2 F_{S_k}(u^k)+\xi_k  I\right)d_{S_k} = -\nabla F_{S_k}(u^k).
	\]
	We perform the Armijo line search along the direction $d^k$ to seek
	an appropriate step-size $\alpha_k$, set $x^{k+1}:=x^k+\alpha_k d^k$, and then return to
	the PG step to guarantee that the iterate sequence has a global convergence property.
	
	From the above statement, the iterate sequence $\{x^k\}_{k\in\mathbb{N}}$ of HpgSRN consists of two parts: the iterates given by the PG step and those generated by
	the subspace regularized Newton step. Some switching conditions involving ${\rm sign}(x^k) = {\rm sign}(\overline{x}^k)$ decide which step the next iterate $x^{k+1}$ enters in.
	
	The main contributions of this paper include three aspects:
	\begin{itemize}
		\item [{\bf(i)}] We propose a hybrid of the PG and subspace regularized Newton methods for solving problem (\ref{model}).
		Different from ZeroFPR and BELLA, each iterate of HpgSRN does not necessarily
		perform a second-order step until sufficiently many steps are performed and the computation of the regularized Newton step
		fully exploits the subspace structure, which substantially reduces the computation cost.
		Numerical comparison with ZeroFPR indicates that HpgSRN not only requires
		much less computing time (especially for those problem with $n\gg m$)
		but also yields comparable even better sparsity and objective function values.
		
		\item[{\bf(ii)}] For the proposed HpgSRN, we achieve the global convergence of the iterate
		sequence under the  local Lipschitz continuity of $\nabla^2\!f$ on $\mathbb{R}^m$ (see \cref{ass1}), the KL property of $F$, and a curve-ratio condition for the subspace regularized Newton directions (see  \cref{ass2}). Both Assumptions \ref{ass1} and \ref{ass2} are commonly used in the convergence analysis of Newton-type methods with line search.
			
			
			\item[\bf(iii)] Under Assumptions \ref{ass1} and \ref{ass2}, if the KL property of $F$ is
			strengthened to be the KL property of exponent $1/2$, we establish the $R$-linear convergence rate of the iterate sequence. If in addition a local error bound condition holds at the limit point,
			the iterate sequence is shown to converge superlinearly with rate $1\!+\!\sigma$ for $\sigma\in(0,1/2]$. This not only removes the local optimality of the limit point
			as required by ZeroFPR and BELLA, but also gets rid of its isolatedness as BELLA does.
		\end{itemize}
		
		The rest of this paper is organized as follows. Section \ref{sec2} gives
		some preliminaries, including the subdifferential characterization of $F$
		and the equivalence between the KL property of exponent $1/2$ of $F$ and that of $F_{S}$.
		Section \ref{sec3} presents the formal iterate steps of HpgSRN and some auxiliary results. Section \ref{sec4} provides the global and local convergence analysis of HpgSRN. Finally, in section \ref{sec5} we conduct numerical experiments for HpgSRN on $\ell_q$-norm regularized linear and logistic regressions on real data and compare its performance with
		ZeroFPR and the PG method with a monotone line search (PGls). 
		
		\subsection{Notation}\label{sec1.3}
		
		Throughout this paper, $\mathbb{R}^n$ denotes the $n$-dimensional Euclidean space,
		equipped with the standard inner product $\langle\cdot,\cdot\rangle$ and its induced
		norm $\|\cdot\|$, and ${\bf B}$ denotes the unit ball. $I$ is an identity matrix whose dimension is known from the context.
		For an integer $k\ge 1$, write $[k]\!:=\{1,\ldots,k\}$;
		and for integers $k_2>k_1\ge 1$, write $[k_1,k_2]\!:=\{k_1,\ldots,k_2\}$.
		For a symmetric matrix $H$, $\lambda_{\min}(H)$ and $\|H\|$
		denote the smallest eigenvalue and the spectral norm of $H$ respectively,
		and $H\succ 0$ means that $H$ is positive definite. For a closed set
		$C\subseteq\mathbb{R}^n$, $\Pi_{C}$ denotes the projection operator onto
		$C$, ${\rm dist}(x,C)$ denotes the Euclidean distance from a point
		$x\in\mathbb{R}^n$ to $C$.
		The notation $\circ$ means the Hadamard product operation
		of vectors. For $x\in \mathbb{R}^n$, ${\rm supp}(x):=\{i\in[n]\,|\, x_i\ne 0\}$ denotes
		its support, ${\rm sign}(x)$ denotes the vector with $[{\rm sign}(x)]_i ={\rm sign}(x_i)$, and $|x|_{\min}\!:=\!\min_{i\in {\rm supp}(x)}|x_i|$ denotes the smallest absolute value of the nonzero entries of $x$. For an index set $S\subseteq [n]$, write $S^{c}\!:=[n]\backslash S$ and denote by $x_{S}\in\mathbb{R}^{|S|}$ the vector
		consisting of those $x_j$ with $j\in S$, and for a matrix $A\in\mathbb{R}^{m\times n}$,
		$A_{S}\in\mathbb{R}^{m\times|S|}$ means the matrix consisting of those columns $A_j$ with $j\in S$. For $a\in\mathbb{R}$, $a_{+}:=\max\{a,0\}$.
		\section{Preliminaries}\label{sec2}
		
		In this section, we recall some necessary concepts and present some preliminary results.  First, we recall the outer semicontinuity (see \cite[Definition 5.4]{RW09}) and the upper semicontinuity (see \cite[p. 266]{BS00}) of a set-valued mapping.
		\begin{definition}\label{outer-continuity}
			A set-valued mapping $\mathcal{F}:\mathbb{R}^n \rightrightarrows \mathbb{R}^m$ is outer semicontinuous at $\overline{x}$ if $\limsup_{x\rightarrow\overline{x}} \mathcal{F}(x) \subseteq \mathcal{F}(\overline{x})$, and is upper semicontinuous at $\overline{x}$ if for any neighborhood $V$ of $\mathcal{F}(\overline{x})$, there exists a neighborhood $U$ of $\overline{x}$ such that for every $x\in U$, $\mathcal{F}(x) \subseteq U$.
		\end{definition}
		
	
	For a proper lower semicontinuous (lsc) function $h\!:\mathbb{R}^n\to(-\infty,\infty]$,
	its proximal mapping associated to parameter $\mu>0$ is defined by
	\[
	{\rm prox}_{\mu h}(x):=\mathop{\arg\min}_{z\in\mathbb{R}^n}\Big\{\frac{1}{2\mu}\|z-x\|^2 + h(z)\Big\}
	\quad{\rm for}\ x\in\mathbb{R}^n.
	\]
	For the proximal mapping of $g$, from \cite[Theorem 2.1]{Chen10} we have the following result.
	\begin{lemma}\label{prox-bound}
		Fix any $\mu>0$ and $y\in\mathbb{R}^n$, if
		$\overline{y}\in {\rm prox}_{\mu g}(y)$,
		then it holds that $|\overline{y}|_{\rm min}\ge\big[\mu q(1\!-\!q)\big]^{\frac{1}{2-q}}$.
	\end{lemma}
	\subsection{Stationary point of problem \eqref{model}}\label{sec2.2}
	Before introducing a stationary point of \eqref{model}, we characterize the  subdifferentials of $g$.  By the expression of $g$ and \cite[Definition 8.3]{RW09}, it is easy to verify that the following result holds,  where the notions of (subdifferential) regularity and horizon cone can be found in \cite[Definition 7.25 \& 3.3]{RW09}.
	\begin{lemma}\label{subdiff-g}
		Fix any $x\in\mathbb{R}^n$. Then, $\widehat{\partial}g(x)=\partial g(x)
		=T_1(x_1)\times\cdots\times T_n(x_n)$ with $T_i(x_i)=\{q{\rm sign}(x_i)|x_i|^{q-1}\}$ if $x_i\ne 0$, otherwise $T_i(x_i)=\mathbb{R}$, and $\partial^{\infty}g(x)=[\widehat{\partial}g(x)]^{\infty}$,
		where  $\widehat{\partial}g(x),\partial g(x)$ and $\partial^{\infty}g(x)$ denote the regular, limiting (or Mordukhovich) and horizontal subdifferential of $g$ at $x$ respectively, and $[\widehat{\partial}g(x)]^{\infty}$ denotes the horizon cone of $\widehat{\partial}g(x)$. Consequently, $g$ is a regular function.
	\end{lemma}
	
	Recall that $f$ is continuously differentiable. By combining Lemma \ref{subdiff-g} and \cite[Exercise 8.8]{RW09}, $\widehat{\partial} F(x) = \partial F(x)$ for all $x\in\mathbb{R}^n$ and function $F$ is regular.  
	In the sequel, we call a vector $x\in\mathbb{R}^n$ a critical point if $0\in\partial F(x)$, and we denote by ${\rm crit}F$ the set of the critical points of $F$.
	Next, we introduce a class of $L$-type stationary points for \eqref{model}.
	\begin{definition}\label{L-stationary}
		\cite{Beck19}	A vector $x\in\mathbb{R}^n$ is called an $L$-type stationary point of problem \eqref{model} if there exists a constant $\mu>0$ such that
		$x\in{\rm prox}_{\mu^{-1}(\lambda g)}(x-\!\mu^{-1}\nabla \psi(x))$, 
		and it is called an $\epsilon$-approximate $L$-type stationary point of  \eqref{model} if there exists a constant $\mu>0$ such that
		$\min_{z\in {\rm prox}_{\mu^{-1}(\lambda g)}(x-\mu^{-1}\nabla \psi(x))}
		\mu\|z-x\|_{\infty} \leq \epsilon$.
	\end{definition}
	\begin{remark}
		It was shown in \cite{Beck19} that any optimal solution of minimizing a ${\cal C}^{1,1}$ function with group sparsity expression as a constraint or a penalty (or both) is an $L$-type stationary point. For \cref{model}, we claim that the set of $L$-type stationary points coincides with that of critical points. Obviously, the set of $L$-type stationary points are contained in that of critical points, so it suffices to argue that the converse inclusion holds. Pick any $\overline{x} \in {\rm crit}  F$. Define $\widetilde{g}(y):=\lambda g(y+\overline{x}) + \langle \nabla\!\psi(\overline{x}),y+\overline{x} \rangle$ for $y\in\mathbb{R}^n$. Since $\lambda g$ is prox-regular at $\overline{x}$ for $-\nabla\!\psi(\overline{x})$ by \cite[Example 2.3]{ochs2018}, the function $\widetilde{g}$ is prox-regular at $0$ for $0$. Since $\widetilde{g}$ is also prox-bounded\footnote{For the definitions of  prox-boundedness and prox-regularity, see \cite[Definitions 1.23 \& 13.27]{RW09}.} with the threshold of prox-boundedness being $\infty$, by \cite[Proposition 8.46 (f)]{RW09} the subgradient inequalities in the definition of prox-regularity can be taken to be global. That is, there exists $\gamma_0 > 0$ such that 
		$\widetilde{g}(y) >\widetilde{g}(0) - \frac{\gamma_0}{2}\|y\|^2 \ {\rm for \ all \ } y \neq 0,$
		which implies that for all $y\neq 0$ and $\gamma > \gamma_0$, 
		$\lambda g(y+\overline{x}) + \frac{\gamma}{2}\|y+\overline{x}-(\overline{x} - \frac{1}{\gamma}\nabla \psi(\overline{x}))\|^2 > \lambda g(\overline{x}) + \frac{\gamma}{2}\| \overline{x}-(\overline{x} - \frac{1}{\gamma}\nabla \psi(\overline{x}))\|^2.$
		Therefore, $\overline{x}$ is the unique minimizer of $\lambda g(\cdot) + \frac{\gamma}{2}\|\cdot - (\overline{x} - \frac{1}{\gamma}\nabla\psi(\overline{x}))\|^2$, which by Definition \ref{L-stationary} means that $\overline{x}$ is an $L$-type stationary point of \cref{model}.
	\end{remark}
	Next, we state
	some differential properties of $F$ in a subspace.
	\begin{lemma}\label{property-FS}
		For the objective function $F$ of \eqref{model}, the following statements hold.
		\begin{itemize}
			\item [(i)] For any given index set $S\subseteq [n]$ and any given  $x\in\mathbb{R}^n\backslash \{0\}$
			with ${\rm supp}(x)\!=S$, the function $F_{S}$ is twice continuously
			differentiable at $x_{S}$ with
			\begin{subequations}
				\begin{align}
					\label{forder-derive}
					\nabla F_{S}(x_{S})=A_{S}^{\mathbb{T}}\nabla\!f(A_{S}x_{S})+\lambda q{\rm sign}(x_{S})\circ|x_{S}|^{q-1},\quad\\
					\label{sorder-derive}
					\nabla^2\!F_{S}(x_{S})=A_{S}^{\mathbb{T}}\nabla^2\!f(A_{S}x_{S})A_{S}
					+\lambda q(q\!-\!1){\rm Diag}(|x_{S}|^{q-2}),
				\end{align}
			\end{subequations}
			and the function $g_{S}$ is three times continuously differentiable at $x_{S}$ with
			\begin{equation}\label{three-deriveg}
				D^3g_{S}(x_{S})(w)
				=q(q\!-\!1)(q\!-\!2){\rm Diag}({\rm sign}(x_{S})\circ|x_{S}|^{q-3}\circ w)
				\quad\forall w\in\mathbb{R}^{|S|}.
			\end{equation}
			
			\item[(ii)] For any given bounded set $C\subseteq\mathbb{R}^n$ and any given constant $\kappa>0$,
			there exist $\widehat{c}_{1}\!>0, \widehat{c}_{2}>0$ and $\widehat{c}_{3}\!>0$
			such that for all  $x\in C\backslash \{0\}$ with $|x|_{\min}\ge\kappa$,
			\begin{align*}
				\|\nabla F_{{\rm supp}(x)}(x_{{\rm supp}(x)})\|\le \widehat{c}_{1},\
				\|\nabla^2F_{{\rm supp}(x)}(x_{{\rm supp}(x)})\|\le \widehat{c}_2,\\
				\|D^3g_{{\rm supp}(x)}(x_{{\rm supp}(x)})\|\le \widehat{c}_3.\qquad\qquad\qquad
			\end{align*}
			
			\item [(iii)] For any $x\in\mathbb{R}^n \backslash \{0\}$, ${\rm dist}(0,\partial F(x))
			=\|\nabla F_{{\rm supp}(x)}(x_{{\rm supp}(x)})\|$.
		\end{itemize}
	\end{lemma}
	\begin{proof}
		Since parts (i) and (ii) are straightforward, we only prove part (iii).
		Fix any $x\in\mathbb{R}^n\backslash \{0\}$. Write $S={\rm supp}(x)$.
		From Lemma \ref{subdiff-g} and \cite[Exercise 8.8]{RW09},
		\[
		\partial F(x)=A^{\mathbb{T}}\nabla\!f(Ax)+T_1(x_1)\times\cdots\times T_n(x_n)
		\]
		where $T_i(x_i)$ has the same expression as in Lemma \ref{subdiff-g}.
		Then, we get ${\rm dist}(0,\partial F(x))=\|A_{S}^{\mathbb{T}}\nabla\!f(A_{S}x_{S})
		+\lambda q\,{\rm sign}(x_S)\circ|x_S|^{q-1}\|$. Together with \eqref{forder-derive},
		the result follows.
	\end{proof}
	\subsection{Kurdyka-{\L}ojasiewicz property}\label{sec2.3}
	The past ten years have witnessed the significant role of the Kurdyka-{\L}ojasiewicz (KL)
	property of the objective function in the convergence analysis of first-order algorithms
	for nonconvex and nonsmooth optimization problems (see, e.g., \cite{Attouch10,Attouch13,Bolte14}).
	Next we recall its definition and establish the equivalence between the KL property of exponent $1/2$ of $F$
	and that of $F_{S}$.
	\begin{definition}\label{KL-Def}
		A proper extended real-valued function $h\!:\mathbb{R}^n\to(-\infty,\infty]$ is said to
		have the KL property at a point $\overline{x}\in{\rm dom} \partial h$ if
		there exist $\eta\in(0,\infty]$, a neighborhood $\mathcal{U}$ of $\overline{x}$,
		and a continuous concave function $\varphi\!:[0,\eta)\to\mathbb{R}_{+}$ satisfying
		\begin{equation}\label{vphi-cond}
			\varphi(0)=0,\varphi\ {\rm is\ continuously\ differentiable\ on}\ (0,\eta)\ {\rm and}\
			\varphi'(s)>0,\ \forall s\in(0,\eta)
		\end{equation}
		such that for all $x\!\in\mathcal{U}\cap\big\{x\in\!\mathbb{R}^n\,|\,h(\overline{x})<h(x)<h(\overline{x})+\eta\}$,
		$$\varphi'(h(x) - h(\overline{x})) {\rm dist}(0,\partial h(x))\!\ge 1.$$
		If $\varphi$ can be chosen as $\varphi(s)=c\sqrt{s}$ for some constant $c>0$, then $h$ is said to have the KL property	of exponent $1/2$ at $\overline{x}$.
		If $h$ has the KL property of exponent $1/2$ at each point of ${\rm dom}\,\partial h$,
		then $h$ is called a KL function of exponent $1/2$.
	\end{definition}
	\begin{remark}\label{KL-remark}
		By \cite[Lemma 2.1]{Attouch10} a proper lsc function $h\!:\mathbb{R}^n\to(-\infty,\infty]$
		has the KL property of exponent $1/2$ at all noncritical points. Thus, to show that
		it is a KL function of exponent $1/2$, it suffices to check its KL property
		of exponent $1/2$ at critical points.
		On the calculation of KL exponent, we refer the readers to the recent works \cite{lipong18,wu21}.
	\end{remark}
	
	The following proposition establishes the equivalence between the KL property of
	exponent $1/2$ of $F$ and that of $F_S$.
	\begin{proposition}\label{prop-FS-KL}
		For any given $\overline{x}\in\mathbb{R}^n  \backslash \{0\}$,
		$F$ has the KL property of exponent $1/2$ at $\overline{x}$ if and only if
		$F_{\overline{S}}$ with $\overline{S}= {\rm supp}(\overline{x})$ has
		the KL property of exponent $1/2$ at $\overline{u}=\overline{x}_{\overline{S}}$.
	\end{proposition}
	\begin{proof}
		From Lemma \ref{subdiff-g} and Lemma \ref{property-FS} (iii), one can verify that  if $\overline{x}\in\mathbb{R}^n  \backslash \{0\}$, $\overline{x}\in{\rm crit}F$ if and only if $\overline{x}_{\overline{S}}\in{\rm crit}F_{\overline{S}}$
		for $\overline{S}={\rm supp}(\overline{x})$. Then, by \cref{KL-remark},
		it suffices to consider the case that $\overline{x}\in{\rm crit}F \backslash \{0\}$.
		
		Necessity. Since $F$ has the KL property of exponent $1/2$ at $\overline{x}$,
		there exist $\eta>0,\varepsilon>0$ and $c>0$ such that
		for all $x\in\Gamma(\varepsilon,\eta)\!:=\big\{x\in\mathbb{R}^n\,|\,\|x-\overline{x}\|\leq\varepsilon,
		F(\overline{x})<F(x) <F(\overline{x})+\eta\big\}$,
		\begin{equation}\label{KL-ineq1}
			{\rm dist}(0,\partial F(x))\ge c\sqrt{F(x)-F(\overline{x})}.
		\end{equation}
		Since $\overline{x}_i\ne 0$ for each $i\in\overline{S}$,
		there exists $\varepsilon'>0$ such that for all $z\in\mathbb{B}(\overline{x},\varepsilon')$,
		$z_i\ne 0$ with each $i\in\overline{S}$. Set $\widetilde{\varepsilon}\!:=\min\{\varepsilon,\varepsilon'\}$.
		Pick any $u\in\Gamma_{\overline{S}}(\widetilde{\varepsilon},\eta)
		\!:=\!\big\{u\in\mathbb{R}^{|\overline{S}|}\,|\,\|u-\overline{u}\| \le\widetilde{\varepsilon},
		F_{\overline{S}}(\overline{u})< F_{\overline{S}}(u)\!<F_{\overline{S}}(\overline{u})+\eta\big\}$.
		Let $x\in\mathbb{R}^n$ with $x_{\overline{S}}=u$ and $x_{\overline{S}^c}=0$.
		Clearly, ${\rm supp}(x)={\overline{S}}$. From \cref{property-FS} (iii),
		it follows that ${\rm dist}(0,\partial F(x))=\|\nabla F_{\overline{S}}(u)\|$.
		Also, from $F_{\overline{S}}(u)=F(x)$ and $F_{\overline{S}}(\overline{u})=F(\overline{x})$,
		we have $x\in\Gamma(\varepsilon,\eta)$. Along with \eqref{KL-ineq1}, we get
		\[
		\|\nabla F_{\overline{S}}(u)\| = {\rm dist}(0,\partial F(x))\ge c\sqrt{F(x)-F(\overline{x})}
		=c\sqrt{F_{\overline{S}}(u)-F_{\overline{S}}(\overline{u})}.
		\]
		By the arbitrariness of $u$ in $\Gamma_{\!{\overline{S}}}(\varepsilon,\eta)$,
		$F_{\overline{S}}$ has the KL property of exponent $1/2$ at $\overline{u}$.
		
		Sufficiency. Since $F_{\overline{S}}$ has the KL property of exponent $1/2$ at $\overline{u}$,
		there are $\widetilde{\varepsilon}>0,\widetilde{\eta}>0,c>0$ such that for all
		$u\in\Gamma_{\overline{S}}(\widetilde{\varepsilon},\widetilde{\eta})
		\!:=\!\big\{u\in\mathbb{R}^{|\overline{S}|}\,|\,\|u-\overline{u}\|\leq\widetilde{\varepsilon},
		F_{\overline{S}}(\overline{u})<F_{\overline{S}}(u) <F_{\overline{S}}(\overline{u})+\widetilde{\eta}\big\}$,
		\begin{equation*}
			{\rm dist}(0,\partial F_{\overline{S}}(u))\geq c\sqrt{F_{\overline{S}}(u)-F_{\overline{S}}(\overline{u})}.
		\end{equation*}
		Since every entry of $\overline{u}$ is nonzero, by reducing $\widetilde{\varepsilon}$
		if necessary, for any $u$ with $\|u-\overline{u}\|\leq\widetilde{\varepsilon}$, its entries are all nonzero. By Lemma \ref{property-FS} (iii),
		the last inequality can be rewritten as
		\begin{equation}\label{sm1}
			\|\nabla F_{\overline{S}}(u)\|\geq  c\sqrt{F_{\overline{S}}(u)-F_{\overline{S}}(\overline{u})}.
		\end{equation}
		By continuity, there exists $\varepsilon'>0$ such that
		for all $x\in\mathbb{B}(\overline{x},\varepsilon')$, ${\rm supp}(x)\supseteq\overline{S}$.
		Let $\widehat{\varpi}\!:=\!{\displaystyle\max_{\|x-\overline{x}\|\le 1}}\|A^{\mathbb{T}}\nabla\!f(Ax)\|_{\infty}$.
		Set $\varepsilon\!:=\min\big\{\frac{1}{4},\widetilde{\varepsilon},\varepsilon',\big(\frac{\widehat{\varpi}+1}
		{\lambda q}\big)^{\frac{1}{q-1}}\big\}$ and $\eta\!:=\!\frac{1}{2}\min\{\widetilde{\eta},1\}$.
		Let $T_1(\varepsilon,\eta):=\big\{x\in\Gamma(\varepsilon,\eta)\,|\,{\rm supp}(x)=\overline{S}\big\}$
		where $\Gamma(\varepsilon,\eta)$ is defined as above, and
		$T_2(\varepsilon,\eta)\!:=\Gamma(\varepsilon,\eta)\backslash T_1(\varepsilon,\eta)$.
		Pick any $x\in\Gamma(\varepsilon,\eta)$. We proceed the proof by two cases.
		
		\noindent
		{\bf Case 1: $x\in T_1(\varepsilon,\eta)$}. Let $u=x_{\overline{S}}$.
		We have $u\in\Gamma_{\overline{S}}(\varepsilon,\eta)\subseteq \Gamma_{\overline{S}}(\widetilde{\varepsilon}, \widetilde{\eta})$, where the second inclusion
		is due to $\varepsilon < \tilde{\varepsilon}$ and $\eta < \tilde{\eta}$.
		From \cref{property-FS} (iii) and \eqref{sm1},
		\[
		{\rm dist}(0,\partial F(x))=\|\nabla F_{\overline{S}}(u)\|
		\geq c\sqrt{F_{\overline{S}}(u)-F_{\overline{S}}(\overline{u})}
		= c \sqrt{F(x)-F(\overline{x})}.
		\]
		
		\noindent
		{\bf Case 2: $x\in T_2(\varepsilon,\eta)$}. Recall that ${\rm supp}(x)\supseteq\overline{S}$.
		By the definition of $T_2(\varepsilon,\eta)$, there exists $i \notin\overline{S}$
		such that $0<|x_i|\le\varepsilon$. Write $S\!:={\rm supp}(x)$.
		Since $F_{S}$ is continuously differentiable at $x_{S}$ by \cref{property-FS} (i),
		for all $i\in S\backslash\overline{S}$ it holds that
		\begin{equation}\label{sm2}
			\begin{aligned}
				{\rm dist}(0,\partial F(x))&\ge |[\nabla F_S(x)]_i|
				=\big|A_i^{\mathbb{T}}\nabla\!f(Ax) + \lambda q{\rm sign}(x_i)|x_i|^{q-1}\big| \\
				&\ge \lambda q |x_i|^{q-1} -\big|A_i^{\mathbb{T}}\nabla\!f(Ax)\big|
				>\lambda q\varepsilon^{q-1}-\widehat{\varpi}\ge1 ,
			\end{aligned}
		\end{equation}
		where the last inequality follows by the definition of $\varepsilon$.
		Since $F(\overline{x})<F(x)<F(\overline{x})+\eta$ and $0<\eta<1$,
		we have $\sqrt{F(x)-F(\overline{x})}<1$. Together with \eqref{sm2}, we have
		\[
		{\rm dist}(0,\partial F(x))>\sqrt{F(x)-F(\overline{x})}.
		\]
		From the above two cases and the arbitrariness of $x$ in $\Gamma(\varepsilon,\eta)$,
		the function $F$ has the KL property of exponent $1/2$ at $\overline{x}$.
		Thus, the proof is completed.
	\end{proof}
	\section{A hybrid of PG and subspace regularized Newton methods}\label{sec3}
	
	In this section, we describe the iterate steps of HpgSRN, a hybrid of PG
	and subspace regularized Newton methods for solving problem \cref{model}. First, we introduce the basic iterate of
	the PG method with a monotone line search (PGls), i.e., a monotone version of SpaRSA \cite{Wright09}. Let $x\in\mathbb{R}^n$ be the current iterate and $\mu>0$ be an initial step-size. One step of the PGls returns a new iterate $x^{+}$ and the used step-size $\mu_{+}$
	such that $F(x^{+})$ has a certain decrease.
	\begin{algorithm}[H]
		\caption{($[x^{+},\mu_{+}]:=\mathcal{G}(x,\mu;\widetilde{\tau},\widetilde{\alpha},\lambda)$)}\label{gist}
		\textbf{Input:} $x\in\mathbb{R}^n$ and parameters $\mu>0,\widetilde{\tau}>1$ and $\widetilde{\alpha}>0$.
		
		Let $l=0,\mu_{l} = \mu$ and $ x^l \in {\rm prox}_{\mu_{l}^{-1}(\lambda g)}\big(x-\mu_{l}^{-1}\nabla \psi(x)\big)$.
		
		\textbf{while}\ $F(x^l)>F(x)-({\widetilde{\alpha}}/{2})\|x^l-x\|^2$
		\begin{itemize}
			\item Let $\mu_{l+1}=\widetilde{\tau}\mu_{l}$ and $l\gets l+1$;
			
			\item Seek $x^l\in {\rm prox}_{\mu_{l}^{-1}(\lambda g)}\big(x-\mu_{l}^{-1}\nabla\psi(x)\big)$;
		\end{itemize}
		\textbf{end (while)}
		
		\noindent
		Let $x^{+} = x^l$ and $\mu_{+} = \mu_{l}$.
	\end{algorithm}
	\begin{remark}\label{remark-Alg1}
		We claim that the number of backtrackings of \cref{gist} is finite. For this purpose, define $\widetilde{h}_{\mu,x}(z):=\langle \nabla \psi(x), z-x\rangle + \frac{\mu}{2}\|z - x\|^2 + \lambda g(z)$ for $z\in\mathbb{R}^n$.
		For each $l \in \mathbb{N}$,  from $\mu_l \geq \mu$ and $x^l\in{\rm prox}_{\mu_{l}^{-1}(\lambda g)}\big(x-\mu_{l}^{-1}\nabla\psi(x)\big)$,
		it follows that
		\begin{equation}\label{hmu-ineq0}
			\widetilde{h}_{\mu,x}(x^l)
			\leq \langle \nabla \psi(x), x^l-x\rangle +\frac{\mu_l}{2}\|x^l - x\|^2
			+\lambda g(x^l) \leq \lambda g(x) = \widetilde{h}_{\mu,x}(x).
		\end{equation}
		Since $\widetilde{h}_{\mu,x}$ is continuous and coercive, the set $\mathcal{L}_{\widetilde{h}_{\mu,x}}:= \{z\in\mathbb{R}^n\ | \ \widetilde{h}_{\mu,x}(z) \leq \widetilde{h}_{\mu,x}(x)\}$ is compact.
		Since $\nabla\psi$ is continuously differentiable,
		there exists $L_{x} > 0$ such that for any $y, w\in \mathcal{L}_{\widetilde{h}_{\mu,x}}$, $\|\nabla \psi(y) - \nabla \psi(w)\|\leq L_x\|y-w\|$. When $\mu_l \geq L_x + \widetilde{\alpha}$, from $x,x^l \in \mathcal{L}_{\widetilde{h}_{\mu,x}}$ and the descent lemma \cite[Proposition A.24]{NP97},
		\begin{align}\label{ineq-Alg1}
			F(x^l)&\le \psi(x)+\langle\nabla\psi(x), x^l\!-\!x\rangle
			+\frac{L_x}{2}\|x^l\!-\!x\|^2 + \lambda g(x^l)\nonumber \\
			&\le\psi(x)+ \langle\nabla\psi(x),x^l\!-\!x\rangle
			+ \frac{\mu_l}{2}\|x^l\!-\!x\|^2 + \lambda g(x^l) -\frac{\widetilde{\alpha}}{2}\|x^l\!-\!x\|^2\nonumber \\
			&\le \psi(x)+ \lambda g(x) -\frac{\widetilde{\alpha}}{2}\|x^l\!-\!x\|^2=F(x)-\frac{\widetilde{\alpha}}{2}\|x^l\!-\!x\|^2,
		\end{align}
		where the last inequality is due to \eqref{hmu-ineq0}.
		This implies that the line search procedure stops in the $l$th backtracking. The above arguments only use the Lipschitz continuity of $\nabla\!\psi$ on the set $\mathcal{L}_{\widetilde{h}_{\mu,x}}$ rather than its global Lipschitz continuity, and the coercivity of $\widetilde{h}_{\mu,x}$ rather than that of $g$. For more discussion on line search of PG methods in a general setting, see also \cite{bello16,salzo17} for the convex $\psi$ and \cite{de21} for the nonconvex $\psi$.
	\end{remark}
	
	Now we are in a position to present the detailed iterate steps of our HpgSRN.
	\begin{algorithm}[!ht]
		\caption{(a hybrid of PG and subspace regularized Newton methods)}\label{hybrid}
		\textbf{Initialization:} Choose $\widetilde{\tau}\!>1,\widetilde{\alpha}>0,$ $\mu_{\max}>\mu_{\min}> 0, \sigma\in(0,\frac{1}{2}],\varrho\in(0,\frac{1}{2}),
		\beta\in(0,1)$, $b_1 > 1$ and $b_2 > 0$.
		Choose an initial $x^0\in\mathbb{R}^n$ and  a tolerance  $\epsilon \geq 0$. Let $k\!=0$. \\
		
		\textbf{Step 1: proximal gradient step}	
		\begin{enumerate}
			
			\item[(S1)] Choose an initial step-size $\mu_{k} \in [\mu_{\min}, \mu_{\max}]$. Set $[\overline{x}^k,\overline{\mu}_k]=\mathcal{G}(x^k,\mu_{k};\widetilde{\tau},\widetilde{\alpha},\lambda)$.
			
			\item[(S2)] {\bf If} $ \overline{\mu}_k\|x^k - \overline{x}^k\|_{\infty} \leq \epsilon$, output $x^k$; 
			
			{\bf otherwise} go to (S3).
			
			
			\item[(S3)]
			Let $\overline{\omega}_k=\overline{\mu}_k\!+\!\lambda q(q\!-\!1)|\overline{x}^k|_{\min}^{q-2}$. {\bf If} 
			\begin{equation}\label{if-else}
				{\rm sign}(x^k) = {\rm sign}(\overline{x}^k)\ \ {\rm and}\ \
				\overline{\mu}_k\!+\!\lambda q(q\!-\!1)|x^k|_{\min}^{q-2}\!\ge \frac{1}{2}\overline{\omega}_k,
			\end{equation}
			{\bf then} go to {\bf Step 2};
			
			{\bf otherwise} let $x^{k+1} = \overline{x}^k$ and $k\gets k+1$.  Go to {\bf Step 1}.
		\end{enumerate}
		
		\textbf{Step 2: subspace regularized Newton step}
		\begin{enumerate}
			\item[(S4)]\label{SNT-S4} Let $S_k={\rm supp}(x^k)$ and $u^k=x_{S_k}^k$.
			Seek a subspace Newton direction $\Delta u^k$ by solving
			$G^k\Delta u= -\!\nabla F_{S_k}(u^k)$, where
			$G^k\!=\!\nabla^2 F_{S_k}(u^k)\!+\!(b_1\zeta_k\!+\!b_2\|\nabla F_{S_k}(u^k)\|^{\sigma})I$ with $\zeta_k\!=[-\lambda_{\min}(\nabla^2 F_{S_k}(u^k))]_{+}$.
			Let $d^k_{S_k}=\Delta u^k$ and $d_{S_k^{c}}^k=0$.
			
			\item[(S5)]\label{SNT-S5} Let $m_k$ be the smallest nonnegative integer $m$ such that
			\begin{equation}\label{ls-NT}
				F_{S_k}(u^k\!+\!\beta^m d_{S_k}^k)\le F_{S_k}(u^k) +\varrho\beta^m\langle \nabla F_{S_k}(u^k),d^k_{S_k}\rangle.
			\end{equation}
			
			\item[(S6)] Let $\alpha_k=\beta^{m_k}$ and $x^{k+1}=x^k+\alpha_k d^k$ and $k \gets k+1$. Go to \textbf{Step 1}.
		\end{enumerate}
	\end{algorithm}
	\begin{remark}\label{remark-hybrid}
		{\bf (a)} \cref{hybrid} uses $ \overline{\mu}_k\|x^k-\overline{x}^k\|_{\infty} \leq \epsilon$ as the stopping rule, which by Definition \ref{L-stationary} means that the output $x^k$ is an $\epsilon$-approximate $L$-type stationary point.
		
		\noindent
		{\bf(b)} Every iterate of Algorithm \ref{hybrid} executes Step 1,
		but does not necessarily  perform Step 2. Step 1 aims to ensure the convergence of
		the whole iterate sequence, while Step 2 is a subspace regularized Newton step
		used to enhance the convergence speed whenever the iterates are stable. The first condition in \eqref{if-else}
		aims to check whether the sign supports of the iterates are stable,
		while the second one is to ensure that  $|x^k|_{\min}$ is sufficiently away from $0$; see \cref{prop-xbark}. The switching criterion \eqref{if-else} plays
		a crucial role in the later convergence analysis. 
		When setting $\epsilon = 0$ and \cref{hybrid} generates an infinite sequence,
		we will show in \cref{supp-lemma} that under \cref{ass1}, after a finite number of iterates,
		\cref{hybrid} reduces to a regularized Newton method to minimize  $F_{S_*}$ for some $S_* \subseteq [n]$.
		
		\noindent
		{\bf(c)} 
		We claim that \cref{hybrid} is well defined.
		By \cref{remark-Alg1}, for each $k\in\mathbb{N}$, the line search step in (S1) is well defined. Next we see that when the iteration goes from Step 1 to Step 2, it is necessary that $S_k \not= \emptyset$. As in this case \eqref{if-else} is satisfied, $x^k\not=0$ must hold. If not, by \eqref{if-else}, $\overline{x}^k=0$. So the termination condition in (S2) is satisfied and the algorithm stops. Finally it suffices to argue that 
		the line search in (S5) will terminate after a finite number of backtrackings. By \cref{property-FS} (i), $F_{S_k}$ is continuously differentiable at $u^k$, which along with  $G^k\succ 0$ implies that 
		\begin{equation}\label{descent-direction}
			\langle\nabla F_{S_k}(u^k), d^k_{S_k}\rangle = - \langle G^kd_{S_k}^k, d_{S_k}^k\rangle <0, 
		\end{equation}
		i.e., $d^k_{S_k}$ is a descent direction of $F_{S_k}$ at $u^k$. In addition, $F_{S_k}$ is bounded from below on $\mathbb{R}^{|S_k|}$ because $f$ is bounded from below on $\mathbb{R}^m$. By following the same arguments as those for \cite[Lemma 3.1]{Wright02}, the smallest nonnegative integer $m_k$ satisfying \cref{ls-NT} exists. Therefore, \cref{hybrid} is well defined.

		From the iterate steps of \cref{hybrid},
		the sequence $\{x^k\}_{k\in\mathbb{N}}$ consists of two parts, i.e., $\{x^k\}_{k\in\mathbb{N}}=\{x^k\}_{k\in\mathcal{K}_1}\cup\{x^k\}_{k\in \mathcal{K}_2}$,
		where
		\[
		\mathcal{K}_1\!:=\big\{k\in\mathbb{N}\ |\ x^{k+1}\ {\rm is\ generated\ by\ Step \ 1}\big\}
		\ \ {\rm and}\ \ \mathcal{K}_2:=\mathbb{N}\backslash \mathcal{K}_1.
		\] 
		It is clear now that for $k\in\mathcal{K}_2$, $S_k \not= \emptyset,$ that is, $x^k$ has a nonempty support.
		
		
		\noindent
		{\bf(d)} Although \cref{hybrid} is a hybrid of PG and second-order methods,
		it is not a special case of ZeroFPR \cite{Themelis18} and FBTN \cite{Themelis19} due to
		the following four aspects. Firstly, each iterate of \cref{hybrid} does not
		necessarily perform Newton step, while each iterate of ZeroFPR and
		FBTN must execute a second-order step. Secondly, \cref{hybrid} is using
		the Armijo line search, which is different from the ones used in
		ZeroFPR and FBTN. Let  $F_{\gamma}$ denote the forward-backward envelope of $F$ associated to $\gamma >0$, and $\eta>0$ be a constant related to the (local) Lipschitz constant of $\nabla \psi$. For \eqref{model}, the line search of ZeroFPR
		is to seek the smallest nonnegative integer $m_k$ of
		those $m$'s such that
		\[
		F_{\gamma}(\overline{x}^k + \beta^m\overline{d}^k) -F_{\gamma}(x^k)
		\le -\eta\|x^k-\overline{x}^k\|^2.
		\]
		Then set $x^{k+1}=\overline{x}^k + \beta^{m_k}\overline{d}^k$,
		where $\overline{d}^k$ is a Newton-type direction at $\overline{x}^k$ rather than $x^k$; and the line search of FBTN is to seek the smallest nonnegative integer $m_k$ of those $m$'s such that
		\[
		F_{\gamma}\left((1-\beta^m) \overline{x}^k+\beta^m(x^k+d^k)\right) -F_{\gamma}(x^k) \leq -\eta \|x^k - \overline{x}^k\|^2,
		\]
		and then set $x^{k+1}=(1-\beta^{m_k}) \overline{x}^k+\beta^{m_k}(x^k+d^k)$,
		where $d^k$ is a second-order direction at $x^k$.
		We observe that the decrease of the successive iterates for ZeroFPR and FBTN,
		i.e.  $F_{\gamma}(x^{k+1})-F_{\gamma}(x^{k})$, is controlled by $ -\|x^k-\overline{x}^k\|^2$,
		while the decrease of the successive iterates for Step 2 of \cref{hybrid}, i.e., $F(x^{k+1})-F(x^{k})$, is
		controlled by the curve ratio $\alpha_k\langle \nabla F_{S_k}(u^k),d^k_{S_k}\rangle$.
		Thirdly,  the line search procedures of ZeroFPR and FBTN involve computing the forward-backward  envelope of $F$, which means that prox-gradient evaluations are needed at each backtracking trial and this is not the case for (S5) of  Algorithm 3.2. Finally, the global convergence analysis of ZeroFPR requires its second-order direction $d^k$ to satisfy
		\begin{equation}\label{dboundxkxbar}
			\exists\ {\rm a\ constant}\ \widehat{c}>0\ {\rm such\ that}\
			\|d^k\|\leq \widehat{c} \|x^k - \overline{x}^k\|\ \ {\rm for\ all}\ k,
		\end{equation}
		but now it is unclear whether the regularized Newton direction in (S4) satisfies \eqref{dboundxkxbar} or not. 
		
		\noindent
		{\bf (e)} Our algorithm is similar to the Newton acceleration framework of the PG method proposed in \cite{bareilles20}, which first uses
		the PG method to identify the underlying manifold substructure of \eqref{model} and then accelerates it with a Riemannian Newton method. However, our algorithm is not a special case of this framework due to the following facts. Firstly, similar to ZeroFPR and FBTN, the framework in \cite{bareilles20} executes a Newton step in each iteration. As discussed in part (d), our algorithm  adaptively executes a Newton step by condition \eqref{if-else}, which avoids some unnecessary waste in second-order step. Secondly, the Riemannian Hessian was used to yield the Newton directions in \cite{bareilles20}, while a regularized one is used in our algorithm to yield the Newton directions. Thirdly, a quadratic convergence rate of the iterate sequence was established in \cite{bareilles20} by assuming that the Riemannian Hessian is positive definite at the limit point. However, under weaker conditions we show that the generated sequence is convergent and has a superlinear convergence rate; see \cref{gconverge} and \cref{lemma-newtondir}, respectively.
	\end{remark}
	
	To conduct the convergence analysis of \cref{hybrid} with $\epsilon =0$ in the next section, from now on we assume that $x^k\ne\overline{x}^k$ for all $k$ (if not, \cref{hybrid} yields an $L$-type stationary point within a finite number of steps), i.e., \cref{hybrid} generates an infinite sequence $\{x^k\}_{k\in\mathbb{N}}$. 
	The following lemma shows that  the sequences $\{x^k\}_{k\in\mathbb{N}}$ and $\{\overline{x}^k\}_{k\in\mathbb{N}}$ are bounded, and the sequence $\{\overline{\mu}_k\}_{k\in\mathbb{N}}$ is upper bounded. The latter will be used to derive a uniform lower bound for $|\overline{x}^k|_{\min}$; see \cref{prop-xbark} (i) later. 
	
	\begin{lemma}\label{well-defineness}
		The following assertions hold for $\{x^k\}_{k\in\mathbb{N}}, \{\overline{x}^k\}_{k\in\mathbb{N}}$ and $\{\overline{\mu}_k\}_{k\in\mathbb{N}}$.
		\begin{itemize}
			\item[(i)] The sequence $\{F(x^k)\}_{k\in\mathbb{N}}$ is nonincreasing and convergent, and consequently, 
			$\{x^k\}_{k\in\mathbb{N}}\subseteq \mathcal{L}_{F}(x^0)\!:=\!\{x \in \mathbb{R}^n\,|\,F(x)\le F(x^0)\}$ and $\{\overline{x}^k\}_{k\in\mathbb{N}}\subseteq \mathcal{L}_{F}(x^0)$.

			\item[(ii)] $\{x^k\}_{k\in\mathbb{N}}$ and $\{\overline{x}^k\}_{k\in\mathbb{N}}$ are bounded, the cluster point set of $\{x^k\}_{k\in\mathbb{N}}$, denoted by $\Omega(x^0)$, is nonempty and compact, and $F$  is constant on $\Omega(x^0)$. 	
			
			\item[(iii)] For all $k\in\mathbb{N}$, $\overline{\mu}_k<\widetilde{L}:=\max\{\mu_{\rm max}+1,\widetilde{\tau}(2\widehat{L}+\widetilde{\alpha})\}$, where
			$\widehat{L}$ is the Lipschitz constant of $\nabla\psi$ on the set $\mathcal{L}_{F}(x^0)\!+\!\overline{\tau}{\bf B}$ with
			$\overline{\tau}:=\frac{\tau_0+ \sqrt{\tau_0^2+ 2\widetilde{c}_{\!f}\mu_{\min}}}{\mu_{\min}}$. Here, $\tau_0:=\max_{x\in\mathcal{L}_F(x^0)} \|\nabla \psi(x)\|$ and
			$\widetilde{c}_f =F(x^0)-c_f$.
		\end{itemize}
	\end{lemma}
	\begin{proof}
		{\bf (i)}  Fix any $k\in\mathbb{N}$. When $k\in \mathcal{K}_1$, $x^{k+1}=\overline{x}^k$, and by \cref{gist}, $F(x^{k+1})\le F(x^k)$. When $k\in\mathcal{K}_2$, from \eqref{ls-NT} and  \eqref{descent-direction} it follows that 
		\[
		F_{S_k}(u^{k+1})\le F_{S_k}(u^{k})+\varrho \beta^{m_k}\langle\nabla F_{S_k}(u^k), d_{S_k}^k \rangle\le F_{S_k}(u^{k}),
		\]
		which along with $S_{k+1}\subseteq S_k$ implies that $F(x^{k+1})\le F(x^k)$. The two cases show that $\{F(x^k)\}_{k\in \mathbb{N}}$ is nonincreasing, which along with the lower boundedness of $F$ means that $\{F(x^k)\}_{k\in \mathbb{N}}$ is convergent. The nonincreasing behavior of $\{F(x^k)\}_{k\in \mathbb{N}}$, together with $F(\overline{x}^k) \leq F(x^k)$ for each $k\in\mathbb{N}$, implies that $F(\overline{x}^k) \leq F(x^k) \leq F(x^0)$ for each $k\in\mathbb{N}$, and consequently,  $\{x^k\}_{k\in\mathbb{N}}\subseteq \mathcal{L}_{F}(x^0)$ and $\{\overline{x}^k\}_{k\in\mathbb{N}}\subseteq \mathcal{L}_{F}(x^0)$. 
		
		\noindent
		{\bf(ii)} Since $g$ is coercive and $f$ is lower bounded, the level set $\mathcal{L}_{F}(x^0)$ is compact. By part (i), $\{x^k\}_{k\in\mathbb{N}}$ and $\{\overline{x}^k\}_{k\in\mathbb{N}}$ are bounded, so the set $\Omega(x^0)$ is nonempty. Using the same arguments as in \cite[Lemma 5 (iii)]{Bolte14} yields the compactness of $\Omega(x^0)$. Pick any $x^*\in\Omega(x^0)$. There exists a subsequence $\{x^{k_j}\}_{j\in\mathbb{N}}$ such that $\lim_{j\to\infty}x^{k_j}=x^*$. By the continuity of $F$ and the convergence of $\{F(x^k)\}_{k\in\mathbb{N}}$, we have  $F(x^*)=\lim_{j\to\infty}F(x^{k_j})=F^*$, where $F^*$ is the limit of $\{F(x^k)\}_{k\in\mathbb{N}}$. This means that $F$ is constant on the set $\Omega(x^0)$. 	
		
		\noindent
		{\bf(iii)} Define $K\!:=\{k\in\mathbb{N}\ |\ \overline{\mu}_k>\mu_k \}$. If $K$ is empty, the desired result holds because $\overline{\mu}_k = \mu_k\le\mu_{\rm max}<\widetilde{L}$ for all $k\in\mathbb{N}$, so we assume that $K\ne\emptyset$. We first argue that 
		\begin{equation}\label{xbound}
			\|\widehat{x}^k - x^k\| \leq \overline{\tau}
			\ \ {\rm for\ each}\ k\in K.
		\end{equation}
		To this end, write $\widehat{\mu}_k :=\overline{\mu}_k/\widetilde{\tau}$ and $\widehat{x}^k:= {\rm prox}_{\widehat{\mu}_k^{-1}(\lambda g)}(x^k \!-\!\widehat{\mu}_k^{-1}\nabla \psi(x^k))$ for each $k\in K$. Since $\widehat{\mu}_k<\overline{\mu}_k$, by \cref{gist} we have  
		$F(\widehat{x}^k)>F(x^k)-\frac{\widetilde{\alpha}}{2}\|\widehat{x}^k\!-\!x^k\|^2$, which implies that $\widehat{x}^k \neq x^k$ for each $k\in K$. For each $k\in K$, from the definition of $\widehat{x}^k$, we have
		\begin{equation}\label{lip-eq1}
			\langle \nabla \psi(x^k), \widehat{x}^k - x^k\rangle + \frac{\widehat{\mu}_k}{2}\|\widehat{x}^k - x^k\|^2 +\lambda g(\widehat{x}^k) -\lambda g(x^k)\leq 0.
		\end{equation}
		By using Cauchy-Schwarz inequality and the nonnegativity of $g$, 
		it follows that 
		\begin{align*}
			&\frac{\widehat{\mu}_k}{2} \|\widehat{x}^k - x^k\|^2  \leq \| \nabla \psi(x^k)\| \|\widehat{x}^k - x^k\| + \lambda g(x^k) - \lambda g(\widehat{x}^k) \\
			&\le\| \nabla \psi(x^k)\| \|\widehat{x}^k - x^k\| + F(x^k) - \psi(x^k)\\
			&\leq \| \nabla \psi(x^k)\| \|\widehat{x}^k - x^k\| + F(x^0) - c_f
			\le\tau_0\|\widehat{x}^k - x^k\| +\widetilde{c}_f,
		\end{align*}
		where the third inequality is due to  $F(x^k)\le F(x^0)$ and $\psi(x^k)\ge c_{f}$, and the last one is by the definitions of $\tau_0$ and $\widetilde{c}_{\!f}$. For each $k\in K$, since $\widehat{\mu}_k\geq \mu_k\geq \mu_{\min}$, from the last inequality, $\frac{\mu_{\min}}{2} \|\widehat{x}^k-x^k\|^2-\tau_0\|\widehat{x}^k-x^k\|-\widetilde{c}_f \leq 0$. This, by the definition of $\overline{\tau}$, implies that inequality \eqref{xbound} holds. Now for each $k\in K$, by the mean-value theorem, there exists $\xi^k$ on the line segment connecting $x^k$ and $\widehat{x}^k$ such that $F(\widehat{x}^k) - F(x^k)=\langle \nabla \psi(\xi^k), \widehat{x}^k - x^k\rangle + \lambda g(\widehat{x}^k) -\lambda g(x^k)$. 
		Substituting this equality into \eqref{lip-eq1} and using 
		$F(\widehat{x}^k) - F(x^k) > -\frac{\widetilde{\alpha}}{2}
		\|x^k\!-\!\widehat{x}^k\|^2 $ yields that
		\begin{align*}
			\frac{\widehat{\mu}_k-\widetilde{\alpha}}{2} \|x^k -\widehat{x}^k\|^2
			& <\langle \nabla \psi(\xi^k)-\nabla \psi(x^k),
			\widehat{x}^k-x^k\rangle\\
			& \le\| \nabla \psi(x^k) - \nabla \psi(\xi^k)\| \| \widehat{x}^k - x^k\|.
		\end{align*}
		From part (i) and \eqref{xbound}, $\{x^k\}_{k\in K} \subseteq\mathcal{L}_{F}(x^0)$ and 
		$\{\widehat{x}^k\}_{k\in K}\subseteq\mathcal{L}_{F}(x^0)\!+\!\overline{\tau}{\bf B}$. 
		Hence, $\{\xi^k\}_{k\in K} \subseteq
		\mathcal{L}_{F}(x^0)\!+\!\overline{\tau}{\bf B}$. From the last inequality,
		for each $k\in K$,
		\[
		\frac{\widehat{\mu}_k - \widetilde{\alpha}}{2} \|x^k - \widehat{x}^k\|
		< \| \nabla \psi(x^k) - \nabla \psi(\xi^k)\| \leq \widehat{L}\|x^k - \xi^k\| \leq \widehat{L}\|x^k - \widehat{x}^k\|.
		\]
		Thus, $\widehat{\mu}_k < 2\widehat{L}+\widetilde{\alpha}$ and $\overline{\mu}_k < \widetilde{\tau}(2\widehat{L}+\widetilde{\alpha})$
		for each $k\in K$. The proof is completed.
	\end{proof} 	
	
	For any given $\gamma>0,s\in\mathbb{R}$, define a real-valued function
	\begin{equation}\label{hgams}
		h_{\gamma,s}(t):=\frac{\gamma}{2}(t-s)^2+\lambda|t|^q\ \ {\rm for}\ t\in\mathbb{R}.
	\end{equation}
	It is easy to see that $t=0$ is always a local minimizer of $h_{\gamma,s}$ and that the absolute value of another possible local minimizer is greater than $\overline{\nu}$, where $\overline{\nu}:=\big(\frac{\lambda q(1-q)}{\gamma}\big)^{\frac{1}{2-q}}$.
	In next lemma, we will establish the existence of a uniform lower bound $\varpi$ of $h''_{\gamma,s}$ at its nonzero local minimizer for any $\gamma > 0$ and $s \in \mathbb{R}$. We will show that the existence of such $\varpi$ will ultimately lead to the validity of the second condition of \cref{if-else} for some $k$ in any large enough interval and hence, together with the validity of the first condition of \cref{if-else}, the infinite cardinality of ${\cal K}_2$. Indeed, if for all the integers $k$ in any large enough interval, $x^{k+1}$ is produced by Step 1, then the sufficient decrease property in (S1) of Step 1 implies that
	$$F(x^k) - F(x^{k+1}) \geq \frac{\widetilde{\alpha}}{2} \|x^k - x^{k+1}\|^2 \ (\mbox{with } x^{k+1} = \overline{x}^k).$$
	Summing this up for all such integers, it follows from the lower boundedness of $F$ that $\sum \|x^k - x^{k+1}\|^2$ is bounded. Thus, for some integer $k$, $\|x^k - x^{k+1}\|$ should be sufficiently small.
	By using an integral mean-value theorem, $|x^k|_{\min}^{q-2} - |\overline{x}^k|_{\min}^{q-2}$ is  bounded by $\|x^k - \bar{x}^k\|$. Therefore $|x^k|_{\min}^{q-2} - |\overline{x}^k|_{\min}^{q-2}$ should be sufficiently small. If so, it is true that $\frac{\varpi}{2} + \lambda q(q-1) ( |x^k|_{\min}^{q-2} - |\overline{x}^k|_{\min}^{q-2} ) \geq 0$, which implies that the second condition of \cref{if-else} holds for some integer $k$.

	\begin{lemma}\label{twice-derive}
		For any given $0<\upsilon<M<\infty$, there exists a constant $\varpi\!>0$ such that
		for any $\gamma>0$ and $s\in \mathbb{R}$ with $\overline{t}(\gamma,s) \in \mathop{\arg\min}_{t\in\mathbb{R}} h_{\gamma,s}(t)$ and $|\overline{t}(\gamma,s)|\in [\upsilon,M]$,
		\[
		h_{\gamma,s}''(\overline{t}(\gamma,s))=\gamma+\lambda q(q\!-\!1)|\overline{t}(\gamma,s)|^{q-2}\geq \varpi.
		\]
	\end{lemma}
	\begin{proof}
		Suppose that the conclusion does not hold. Then, there exist sequences $\{\gamma_{k}\}\subset\mathbb{R}_{++}$ and $\{s_{k}\}\subset\mathbb{R}$ with $|\overline{t}(\gamma_{k},s_{k})|\in[\upsilon, M]$ such that $h_{\gamma_k,s_k}''(\overline{t}(\gamma_{k},s_{k}))\le\frac{1}{k}$ for all $k\in\mathbb{N}$. For each $k\in\mathbb{N}$, write $\overline{t}_{k}:=\overline{t}(\gamma_{k},s_{k})$ and $\vartheta_{k}:=h_{\gamma_k,s_k}$. Clearly, there exists $\overline{k}\in\mathbb{N}$ such that for all $k>\overline{k}$, $\vartheta_{k}''(\overline{t}_{k})<\frac{\kappa\upsilon}{10}:=\varepsilon$,  where $\kappa:=\lambda q(q\!-\!1)(q\!-\!2)M^{q-3}$. By the expression of $\vartheta_{k}$, for any $t$ with $|t|\in(0,M]$, the following inequality holds:
		\begin{equation}\label{three-derive}
			|\vartheta_{k}'''(t)|=\lambda q(q\!-\!1)(q\!-\!2) |t|^{q-3}\ge\kappa.
		\end{equation}
		Fix any $k>\overline{k}$. We proceed the arguments by	$\overline{t}_{k}\in[\upsilon,M]$ and $\overline{t}_{k}\in[-M,-\upsilon]$.
		
		\noindent
		{\bf Case 1: $\overline{t}_{k}\in[\upsilon,M]$.} Since $\vartheta_{k}''(\overline{t}_{k})<\varepsilon$	and $\vartheta_{k}'''(t)\!>\kappa$ for $t\in(0,M]$, by the integral mean-value theorem,  $\vartheta_{k}''(\overline{t}_{k})>\vartheta_{k}''(\overline{t}_{k}\!-\!\frac{\varepsilon}{\kappa})+\varepsilon$, which by $\vartheta_{k}''(\overline{t}_{k})<\varepsilon$ implies that $\vartheta_{k}''(\overline{t}_{k}\!-\!\frac{\varepsilon}{\kappa})<0$. Together with 
		$\vartheta_{k}''(\overline{t}_{k}) > 0$ (see \cite[Lemma 14]{Hu17}), there exists $0<\delta<\frac{\varepsilon}{\kappa}$ such that $\vartheta_{k}''(\overline{t}_{k}\!-\!\delta)\!=0$. Recall that  	$\vartheta_{k}'''(t)\!>\kappa > 0$ for all $t\in(0,M]$. Then,
		\begin{equation}\label{fact1}
			\vartheta_{k}''(t)<0\ \ {\rm for}\ t\in(0,\overline{t}_{k}\!-\!\delta)
			\ \ {\rm and}\ \ \vartheta_{k}''(t) > 0\ \ {\rm for} \ t\in (\overline{t}_{k}\!-\!\delta,M].
		\end{equation}
		Note that $\vartheta_{k}'(\overline{t}_{k})=0$. This, along with the second inequality in \cref{fact1}, implies that $\vartheta_{k}'(\overline{t}_{k}\!-\!\delta) < 0$. Also, since $0<\vartheta_{k}''(t)<\varepsilon$ for all $t\in(\overline{t}_{k}\!-\!\delta,\overline{t}_{k})$, from the integral mean-value theorem, $\vartheta_{k}'(\overline{t}_{k}\!-\!\delta)
		>\vartheta_{k}'(\overline{t}_{k})-\varepsilon\delta =-\varepsilon \delta$, and then $\vartheta_{k}'(\overline{t}_{k}\!-\!\delta)\in(-\varepsilon \delta, 0)$.
		Next we argue that there exists a point $\tilde{t}_{k}\in(\overline{t}_{k}\!-\!\delta \!-\!\sqrt{{2\varepsilon \delta}/{\kappa}},\overline{t}_{k}\!-\!\delta)$ such that
		$\vartheta_{k}'(\tilde{t}_{k}) = 0$, which along with the first inequality in  \cref{fact1} implies that
		\begin{equation}\label{fact4}
			\vartheta_{k}'(t) > 0\ {\rm for} \ t\in(0,\tilde{t}_{k})\ \ {\rm and}\ \
			\vartheta_{k}'(t) < 0 \ {\rm for} \ t\in (\tilde{t}_{k}, \overline{t}_{k}-\delta).
		\end{equation}
		Indeed, for any $t\in(0,\overline{t}_{k}\!-\!\delta)$, using $\vartheta_{k}''(\overline{t}_{k}\!-\!\delta) = 0$ and inequality \cref{three-derive} yields that 
		\begin{align*}
			-\varepsilon \delta<\vartheta_{k}'(\overline{t}_{k}\!-\!\delta)
			&=\vartheta_{k}'(t) +\int_{t}^{\overline{t}_{k}\!-\!\delta}\vartheta_{k}''(s)ds
			=\vartheta_{k}'(t) + \int_{t}^{\overline{t}_{k}\!-\!\delta}
			\left[\vartheta_{k}''(s) -\vartheta_{k}''(\overline{t}_{k}\!-\!\delta)\right]ds\\
			&\le\vartheta_{k}'(t) +\int_{t}^{\overline{t}_{k}\!-\!\delta}\kappa (s\!-\!\overline{t}_{k}+\delta)ds
			=\vartheta_{k}'(t)-\frac{\kappa}{2} (t\!-\!\overline{t}_{k}+\delta)^2,
		\end{align*}
		which implies that $\vartheta_{k}'(t)>0$ for all $ t\le\overline{t}_{k}\!-\!\delta \!-\!\sqrt{{2\varepsilon \delta}/{\kappa}}$. Along with $\vartheta_{k}'(\overline{t}_{k}\!-\!\delta) < 0$, there exists $\tilde{t}_{k}\in (\overline{t}_{k}\!-\!\delta\!-\!\sqrt{\frac{2\varepsilon\delta}{\kappa}},\overline{t}_{k}\!-\!\delta)$ such that $\vartheta_{k}'(\tilde{t}_{k}) =0$. 
		
		From \cref{fact1} we deduce that $\vartheta_{k}'$ is decreasing in $(\tilde{t}_{k}, \overline{t}_{k}\!-\!\delta)$ and is increasing in $(\overline{t}_{k}\!-\!\delta,\overline{t}_{k})$, which means that
		$\vartheta_{k}'(t)\ge\vartheta_{k}'(\overline{t}_{k}-\delta)>-\varepsilon \delta$
		for all $t\in(\tilde{t}_{k},\overline{t}_{k})$. Then,
		\begin{align}\label{q1-temp}
			\vartheta_{k}(\overline{t}_{k})-\vartheta_{k}(\tilde{t}_{k})
			&= \int^{\overline{t}_{k}}_{\tilde{t}_{k}}\!\vartheta_{k}'(s)ds
			>-\varepsilon\delta(\overline{t}_{k} - \tilde{t}_{k})
			>-\varepsilon \delta\Big(\delta\!+\!\sqrt{\frac{2\varepsilon \delta}{\kappa}}\Big) \\
			&>-\Big(\frac{\varepsilon^3}{\kappa^2} + \sqrt{2}\frac{\varepsilon^3}{\kappa^2}\Big)
			>- 3\frac{\varepsilon^3}{\kappa^2} =-\frac{3\kappa}{1000}\upsilon^3,\nonumber
		\end{align}
		where the third inequality is due to $0<\delta<\frac{\varepsilon}{\kappa}$.
		On the other hand, we have
		\begin{align}\label{q2-temp}
			\vartheta_{k}(\tilde{t}_{k})-& \vartheta_{k}(0)  
			= \int_{0}^{\tilde{t}_{k}}\!\int_{\tilde{t}_{k}}^s \vartheta_{k}''(\tau)d\tau ds
			=\int_{0}^{\tilde{t}_{k}}\!\int_{\tilde{t}_{k}}^s
			[\vartheta_{k}''(\tau)-\vartheta_{k}''(\overline{t}_{k}\!-\!\delta)]d\tau ds\nonumber\\
			& \geq \int_{0}^{\tilde{t}_{k}} \int_{\tilde{t}_{k}}^s \kappa (\tau-\overline{t}_{k}+\delta) d\tau ds
			=\int_{0}^{\tilde{t}_{k}} \frac{\kappa}{2}s^2 -\frac{\kappa}{2}\tilde{t}_{k}^2
			+\kappa(\overline{t}_{k}\!-\!\delta)(\tilde{t}_{k}\!-\!s) ds \nonumber\\
			&= \frac{\kappa}{6}\tilde{t}_{k}^3 -\frac{\kappa}{2}\tilde{t}_{k}^3
			+ \frac{\kappa\tilde{t}_{k}^2}{2}(\overline{t}_{k}\!-\!\delta) \geq \frac{\kappa}{6}\tilde{t}_{k}^3 \geq
			\frac{\kappa}{6} \Big(\overline{t}_{k}\!-\!\delta\!-\!\sqrt{\frac{2\varepsilon \delta}{\kappa}}\Big)^3
			\geq \frac{\kappa}{6} \Big(\upsilon\!-\!\frac{3\varepsilon}{\kappa}\Big)^3,
		\end{align}
		where the first equality is due to $\vartheta_{k}'(\tilde{t}_{k}) =0$,
		the second one is using $\vartheta_{k}''(\overline{t}_{k}\!-\!\delta)=0$,
		the first inequality is using \cref{three-derive} and the last inequality
		is due to $0<\delta<\frac{\varepsilon}{\kappa}$ and $\overline{t}_{k}\ge\upsilon$.
		Thus, from \cref{q1-temp} and \cref{q2-temp} and $\varepsilon:=\frac{\kappa \upsilon}{10}$, we have	$\vartheta_{k}(\overline{t}_{k})-\vartheta_{k}(0)
		>\frac{465\kappa}{6000}\upsilon^3>0$,
		contradicting that $\overline{t}_{k}$ is a global minimizer
		of $\vartheta_{k}=h_{\gamma_{k},s_{k}}$. The conclusion then holds.
		
		\noindent
		{\bf Case 2:} $\overline{t}_{k}\in[-M,-\upsilon]$. By using the similar arguments to those for Case 1, one can verify that the conclusion holds. Here, the details are omitted.
	\end{proof} 	
	
	To provide a sufficient condition for the switching condition \cref{if-else}, we introduce the following notation that will be used in the subsequent analysis:
	\[
	\overline{S}_k\!:={\rm supp}(\overline{x}^k)\ \ {\rm and}
	\ \ \overline{u}^k\!:=\overline{x}_{\overline{S}_k}^k
	\ \ {\rm for\ each}\ k\in\mathbb{N}.
	\]
	\vspace{-0.6cm}
	\begin{lemma}\label{prop-xbark}
		Let $\{x^k\}_{k\in\mathbb{N}}$ and $\{\overline{x}^k\}_{k\in\mathbb{N}}$ be  generated by \cref{hybrid}, and write $\nu\!:=[\widetilde{L}^{-1}\lambda q(1-q)]^{\frac{1}{2-q}}$. Then, the following statements hold.
		\begin{itemize}
			\item[(i)] $|\overline{x}^k|_{\rm min}>\nu$ for all $k\in\mathbb{N}$, and
			$|x^k|_{\rm min}>\nu$ for all $k\in\mathcal{K}_2$.
			
			\item[(ii)] $\overline{\omega}_k\ge \varpi$ for all $k\in\mathbb{N}$,
			where $\varpi$ is the one in \cref{twice-derive} with $v = \nu$ and $M = \big(\frac{F(x^0)-c_{\!f}}{\lambda}\big)^{\frac{1}{q}}$.
			
			\item[(iii)] For each $k\in\mathbb{N}$, if $|x^k|_{\min}>\frac{\nu}{2}$ and $\|x^k\!-\!\overline{x}^k\| \leq \min\big\{\frac{\nu}{3},\frac{2^{q-3}\varpi}{2\lambda q(1-q)(2-q)\nu^{q-3}}\big\}$, then condition \eqref{if-else} holds.
		\end{itemize}
	\end{lemma}
	\begin{proof}
		{\bf(i)} By using \cref{prox-bound} with $\mu=\overline{\mu}_k^{-1}\lambda$ and $y=x^k-\overline{\mu}_k^{-1}\nabla\!\psi(x^k)$ for each $k\in\mathbb{N}$ and noting that $\mu_{\min}\le\overline{\mu}_k<\widetilde{L}$ from \cref{well-defineness} (iii), we have $|\overline{x}^k|_{\rm min}>\nu$ for all $k$. To argue that $|x^k|_{\min}>\nu$ for all $k\in\mathcal{K}_2$, we only need to prove that $|x^k|_{\min}>\nu$ if $x^k$ satisfies condition \eqref{if-else}. Indeed, the second condition in \eqref{if-else} is equivalent to
		$|x^k|_{\min}^{q-2}\le\frac{\overline{\mu}_k}{2\lambda q(1\!-\!q)}+\frac{1}{2}|\overline{x}^k|_{\min}^{q-2}$,
		which by $\overline{\mu}_k<\widetilde{L}$
		and the definition of $\nu$ means that
		\begin{equation*}
			|x^k|_{\min}^{q-2}<\frac{\widetilde{L}}{2\lambda q(1\!-\!q)}+\frac{1}{2}\nu^{q-2}
			=\nu^{q-2},
		\end{equation*}
		where the equality is using the expression of $\nu$. Thus,  $|x^k|_{\min}>\nu$ for all $k\in\mathcal{K}_2$.
		
		\noindent
		{\bf(ii)} 
		From Algorithm \ref{gist}, $F(\overline{x}^k)\leq F(x^k)$ for each $k\in\mathbb{N}$. Then,
		\[
		c_{f}+\lambda\|\overline{x}^k\|_q^q\le \psi(\overline{x}^k)+\lambda\|\overline{x}^k\|_q^q
		=F(\overline{x}^k)\le F(x^{k})\le F(x^0),
		\]
		which implies that $\|\overline{x}^k\|_q^q\le\lambda^{-1}(F(x^0)-c_{\!f})$,
		and then
		$|\overline{x}_i^k|\le\big(\frac{F(x^0)-c_{\!f}}{\lambda}\big)^{1/q}$ for each $i\in\overline{S}_k$. In addition,
		from part (i), $|\overline{x}_i^k|>\nu$ for each $i\in\overline{S}_k$. For each $k$, let $y^k\!:=x^k\!-\!\overline{\mu}_k^{-1}\nabla\!\psi(x^k)$.
		Then, $\overline{x}_i^k\in\mathop{\arg\min}_{t\in\mathbb{R}}h_{\overline{\mu}_k,y_i^k}(t)$ for each $i\in\overline{S}_k$,
		where $h_{\overline{\mu}_k,y_i^k}$ is defined by \eqref{hgams}. Now by invoking \cref{twice-derive} with $\upsilon=\nu$, $M=\big(\frac{F(x^0)-c_{\!f}}{\lambda}\big)^{\frac{1}{q}}$ and $\overline{t}(\overline{\mu}_k, y^k_i) = \overline{x}^k_i$ for all $i\in\overline{S}_k$, we obtain $\overline{\omega}_k=\overline{\mu}_k+\lambda q(q\!-\!1)|\overline{x}^k|_{\min}^{q-2}\geq \varpi$.

		\noindent
		{\bf (iii)} Fix any $k\in\mathbb{N}$. We first prove that the equality in  \eqref{if-else} holds. From part (i), $|\overline{x}^k|_{\min} > \nu$, while from the given condition, $|x^k|_{\min}>\frac{\nu}{2}$. If there exists an index $i\in[n]$ such that ${\rm sign}(x_i^k) \neq {\rm sign}(\overline{x}^k_i)$,
		then $\|x^k-\overline{x}^k\|\ge|x^k_i-\overline{x}^k_i|>\frac{\nu}{2}$, which is a contradiction to  $\|x^k-\overline{x}^k\|<\nu/3$. Thus, ${\rm sign}(x^k) = {\rm sign}(\overline{x}^k)$, and hence $S_k = \overline{S}_k$.
		For the inequality in \eqref{if-else}, from part (ii), it suffices to argue that
		$\frac{\varpi}{2}+\lambda q(q\!-\!1)(|x^k|_{\min}^{q-2}-|\overline{x}^k|_{\min}^{q-2})
		\ge 0$ or $|x^k|_{\min}^{q-2}-|\overline{x}^k|_{\min}^{q-2}\le\frac{\varpi}{2\lambda q(1-q)}$. Indeed, by invoking the integral mean-value theorem,
		\begin{align*}
			& |x^k|_{\min}^{q-2}\!-\!|\overline{x}^k|_{\min}^{q-2}
			=\!\int^{|\overline{x}^k|_{\min}}_{|x^k|_{\min}} (2\!-\!q) t^{q-3}dt \\
			&\le (2\!-\!q)\big(\min\{|x^k|_{\min},|\overline{x}^k|_{\min}\}\big)^{q-3}
			\big||x^k|_{\min}\!-\!|\overline{x}^k|_{\min}\big| \\
			&<(2\!-\!q)({\nu}/{2})^{q-3}\big||x^k|_{\min}\!-\!|\overline{x}^k|_{\min}\big|\le (2\!-\!q)({\nu}/{2})^{q-3}\|x^k-\overline{x}^k\|
			\le\frac{\varpi}{2\lambda q(1-q)},
		\end{align*}
		where the second inequality is by $|\overline{x}^k|_{\min}>\nu$
		and $|x^k|_{\min}>\frac{\nu}{2}$, the third one is due to $S_k =\overline{S}_k$, and the last one is using
		$\|x^k - \overline{x}^k\|<\frac{2^{q-3}\varpi}{2\lambda q(1-q)(2-q)\nu^{q-3}}$. 
	\end{proof}
	
	From Lemma \ref{prop-xbark}, we obtain the following corollary, stating that $\mathcal{K}_2$ contains infinite indices, so HpgSRN is different from PG method. In the next section, we improve this result so that after a finite number of steps, the iterates of Algorithm \ref{hybrid} always enter into Step 2. 
	\begin{corollary}
		There exists $\overline{k}\in\mathbb{N}$ such that for any
		$k_1,k_2 \in \mathbb{N}$  with $k_2-k_1\!>\overline{k}$, $[k_1,k_2]\cap \mathcal{K}_2\ne\emptyset$, so $\mathcal{K}_2$ is an infinite set and \cref{hybrid} is different from PG method.
	\end{corollary}
	\begin{proof}
		Let $\delta = \min\{\frac{\nu}{3},\frac{2^{q-3}\varpi}{2\lambda q(1-q)(2-q)\nu^{q-3}}\big\}$ and $\overline{k}=\!\lceil\frac{2(F(x^0)-c_f)}{\widetilde{\alpha}\delta^2}\rceil$.
		We argue by contradiction that the result holds. If not, there must exist
		$\widehat{k}_1,\widehat{k}_2\in\mathbb{N}$ with $\widehat{k}_2 -\widehat{k}_1>\overline{k}$
		such that $[\widehat{k}_1,\widehat{k}_2]\cap \mathcal{K}_2=\emptyset$.  Clearly,
		$[\widehat{k}_1,\widehat{k}_2]\subseteq\mathcal{K}_1$. By the definition of $\mathcal{K}_1$,
		for every $k-1\in[\widehat{k}_1,\widehat{k}_2\!-\!1]$, $x^k$ is obtained by the PG step, which by \cref{prop-xbark} (i) implies that $|x^k|_{\min}>\nu$ and then
		$\|\overline{x}^{k}-x^{k}\|\geq \delta$ must hold (if not, by \cref{prop-xbark} (iii), $[\widehat{k}_1\!+\!1,\widehat{k}_2]$ would contain an index of $\mathcal{K}_2$).
		For every $k\in[\widehat{k}_1,\widehat{k}_2]\subset\mathcal{K}_1$,
		we also have $x^{k+1}=\overline{x}^k$. By \cref{gist},
		for every $k\in[\widehat{k}_1,\widehat{k}_2]$,
		$F(x^{k+1})\le F(x^k)-\frac{\widetilde{\alpha}}{2} \|\overline{x}^{k}-x^k\|^2$, and then
		\[
		\frac{2\big(F(x^{\widehat{k}_1+1})-c_f)}{\widetilde{\alpha}}
		\ge\frac{2\big(F(x^{\widehat{k}_1+1})-F(x^{\widehat{k}_2+1})\big)}
		{\widetilde{\alpha}}
		\ge\sum_{i=\widehat{k}_1+1}^{\widehat{k}_2} \|\overline{x}^{k}-x^k\|^2
		\ge(\widehat{k}_2-\widehat{k}_1)\delta^2,
		\]
		where the last inequality is due to $\|\overline{x}^{k}-x^{k}\|\geq \delta$
		for every $k\in[\widehat{k}_1\!+\!1,\widehat{k}_2]$. Together with $F(x^{\widehat{k}_1+1})\le F(x^0)$,
		we obtain $\widehat{k}_2-\widehat{k}_1\leq \frac{2(F(x^0)-c_f)}{\widetilde{\alpha}\delta^2}\le\overline{k}$,
		a contradiction to the given condition $\widehat{k}_2-\widehat{k}_1>\overline{k}$.
		The proof is then completed.
	\end{proof}
	
	\section{Convergence analysis}\label{sec4}
	
	In this part, we analyze the convergence rate of the objective function value sequence
	$\{F(x^k)\}_{k\in\mathbb{N}}$, and establish the global convergence of the iterate
	sequence $\{x^k\}_{k\in\mathbb{N}}$ and its superlinear convergence rate. Throughout this section, we write 
	\[	
	r^k:=\nabla F_{S_k}(u^k)\ \ {\rm and}\ \ H^k:=\nabla^2 F_{S_k}(u^k)
	\ \ {\rm for\ each}\ k\in\mathcal{K}_2.
	\]
	First, we give several technical lemmas that are used for the subsequent convergence analysis. The following lemma states that the subsequences $\{r^k\}_{k\in\mathcal{K}_2}$ and $\{d^k\}_{k\in \mathcal{K}_2}$ are bounded,
	and the subsequence $\{r^k\}_{k\in\mathcal{K}_2}$ is lower bounded by $\{\|u^k-\overline{u}^k\|\}_{k\in\mathcal{K}_2}$.  The latter is crucial to control $F(x^{k+1}) - F(x^{k})$ by using $-\|x^k-\overline{x}^k\|^2$; see \cref{lemma-sequence}.
	\begin{lemma}\label{lemma-bound}
		Let $\{x^k\}_{k\in\mathbb{N}}$ be generated by \cref{hybrid}. The following holds.
		\begin{itemize}
			\item[(i)]  There exists a constant $r_{\!\rm max}>0$ such that $\|r^k\|\le r_{\!\rm max}$ and $\|d^k\|\le b_2^{-1}r_{\!\rm max}^{1-\sigma}$ for all $k\in\mathcal{K}_2$.
			
			\item[(ii)] For each $k\in\mathcal{K}_2$, $\|r^k\|\geq \frac{\varpi}{4}\|u^k\!-\!\overline{u}^k\|$
			where $\varpi$ is the same as in \cref{twice-derive}.
		\end{itemize}
	\end{lemma}
	\begin{proof}
		{\bf(i)} Fix any $k\in\mathcal{K}_2$.  By \cref{remark-hybrid} (c), we know that $S_k \neq \emptyset.$ 
		From \cref{prop-xbark} (i), $|x_i^k|>\nu$ for all $i\in S_k$. By invoking \cref{property-FS} (ii) with $\kappa={\nu}/{2}$
		and $C=\big\{z\in\mathcal{L}_{F}(x^0)\,|\,|z_i|\ge \nu/2\ {\rm for\ all}\ i\in S_k \big\}$, there exists $r_{\!\rm max}>0$ (independent of $k$)
		such that $\|r^k\|\le r_{\!\rm max}$. Together with $\lambda_{\rm min}(G^k)\ge b_2\|r^k\|^\sigma$, it follows that
		\begin{equation}\label{dir-bound}
			\|d^k\|=\|d_{S_k}^k\| \leq \|(G^k)^{-1}\| \|r^k\|\le b_2^{-1}\|r^k\|^{1-\sigma}
			\le b_2^{-1}r_{\!\rm max}^{1-\sigma}.
		\end{equation}
		
		\noindent
		{\bf(ii)} Fix any $k\in\mathcal{K}_2$. Write
		$B_k\!:=A_{S_k}$ and $v^k\!:=u^k\!-\!\overline{\mu}_k^{-1}B_k^{\mathbb{T}}\nabla\! f(B_ku^k)$.
		Let $h_k(u)\!:=\sum_{i=1}^{|S_k|}h_{\overline{\mu}_k,v_i^k}(u_i)$ for $u\in\mathbb{R}^{|S_k|}$,
		where $h_{\overline{\mu}_k,v_i^k}$ is the function defined in \eqref{hgams}
		with $(\gamma,s)=(\overline{\mu}_k,v_i^k)$. From (\ref{if-else}),
		${\rm sign}(x^k) ={\rm sign}(\overline{x}^k)$, and then  $S_k=\overline{S}_k$. Therefore, we have $\overline{u}^k \in \mathop{\arg\min}_{u \in \mathbb{R}^{|S_k|}}h_k(u)$, whose   optimality condition is given by
		\begin{equation}\label{hk-optcond}
			0=\nabla h_k(\overline{u}^k)=\overline{\mu}_k (\overline{u}^k\!-\!v^k)
			+\lambda q {\rm sign}(\overline{u}^k) \circ |\overline{u}^k|^{q-1}.
		\end{equation}
		In addition, by combining \cref{prop-xbark} (ii) and the inequality in (\ref{if-else}), it holds that
		\begin{equation}\label{xk-posdef}
			\varpi/2\le\overline{\omega}_k/2 \le \overline{\mu}_k + \lambda q(q\!-\!1)|x^k|_{\min}^{q-2}
			=\overline{\mu}_k+\lambda q(q\!-\!1)|u^k|_{\min}^{q-2}=h_{\overline{\mu}_k,v_i^k}''(|u^k|_{\min}).
		\end{equation}	
		Define the index sets $\mathcal{I}_{1}^k\!:=\!\{i\in[|S_k|]\,|\, u_i^k\!>0\}$
		and $\mathcal{I}_{2}^k\!:=[|S_k|]\backslash \mathcal{I}_{1}^k$.
		For each $i\in[|S_k|]$, write $\widetilde{u}^k_i:={\rm sign}(u^k_i)\min\{|u^k_i|,|\overline{u}^k_i|\}$.
		Note that each $h_{\overline{\mu}_k,v_i^k}$ is smooth
		at any $t\ne 0$, and $h_{\overline{\mu}_k,v_i^k}''$ is nonincreasing at $(-\infty, 0)$ and nondecreasing at $(0,\infty)$.
		From \eqref{xk-posdef} and \cref{prop-xbark} (ii), it follows that  $h_{\overline{\mu}_k,v_i^k}''(\widetilde{u}_i^k)\ge\varpi/2$ for all $i\in[|S_k|]$.
		Consequently, there exists $\varepsilon>0$ such that for each
		$i\in\mathcal{I}^k_1$, $h_{\overline{\mu}_k,v_i^k}''(t)>\frac{\varpi}{4}$ 
		when $t\in(\widetilde{u}_i^k\!-\varepsilon,\infty)$; and for each $i\in\mathcal{I}^k_2$, $h_{\overline{\mu}_k,v_i^k}''(t)>\frac{\varpi}{4}$ when $t\in(-\infty,\widetilde{u}_i^k\!+\varepsilon)$.
		Define 
		\[
		\Omega_k\!:=\big\{u\in\mathbb{R}^{|S_k|}\,|\,u_i\!>\widetilde{u}_i^k-\varepsilon\ {\rm for}\
		i\in\mathcal{I}_1^k\ {\rm and}\ u_i<\!-\widetilde{u}_i^k +\varepsilon\ {\rm for}\ i\in \mathcal{I}_2^k\big\}.
		\]
		Then, $h_k$ is twice continuously differentiable on the convex set $\Omega_k$
		with $\nabla^2h_k(u)\succ \frac{\varpi}{4}I$ for all $u\in \Omega_k$,
		which implies that $\widetilde{h}_k(u)\!:=h_k(u)-\frac{\varpi}{8}\|u-v^k\|^2$
		is strongly convex on the set $\Omega_k$. From \eqref{hk-optcond} and the expression
		of $\widetilde{h}_k$, clearly, $\nabla\widetilde{h}_k(\overline{u}^k)=\frac{\varpi}{4}(v^k\!-\!\overline{u}^k)$.
		Let $\widehat{u}^k:=u^k+\frac{4}{\varpi}\nabla\widetilde{h}_k(u^k)$.
		By the convexity of $\widetilde{h}_k$ on $\Omega_k$ and $u^k,\overline{u}^k\in\Omega_k$, we have
		\[
		0\le\langle \nabla\widetilde{h}_k(\overline{u}^k)-\nabla\widetilde{h}_k(u^k), \overline{u}^k - u^k\rangle
		= \frac{\varpi}{4}\langle (v^k-\overline{u}^k)-(\widehat{u}^k- u^k),\overline{u}^k - u^k\rangle,
		\]
		which implies that
		$\|u^k-\overline{u}^k\|\le\|v^k-\widehat{u}^k\|=\|u^k-\overline{\mu}_k^{-1}B_k^{\mathbb{T}}\nabla f(B_ku^k)-\widehat{u}^k\|$.
		Together with $\frac{\varpi}{4}(\widehat{u}^k\!-\!u^k)=\nabla\widetilde{h}_k(u^k)
		=(\overline{\mu}_k\!-\!\frac{\varpi}{4})(u^k\!-\!v^k)+\lambda q {\rm sign}(u^k)\circ |u^k|^{q-1}$,
		it follows that
		\begin{align*}
			\|r^k\|&=\big\|B_k^{\mathbb{T}}\nabla\!f(B_ku^k)+\lambda q{\rm sign}(u^k)\circ|u^k|^{q-1}\big\| \\
			&=\big\|B_k^{\mathbb{T}}\nabla\!f(B_ku^k)-(\overline{\mu}_k-\frac{\varpi}{4})(u^k-v^k)+\frac{\varpi}{4}(u^k\!-\!\widehat{u}^k)\big\| \\				 &=\frac{\varpi}{4}\|\overline{\mu}_k^{-1}B_k^{\mathbb{T}}\nabla\!f(B_ku^k)-u^k +\widehat{u}^k\|
			\ge\frac{\varpi}{4}\|u^k - \overline{u}^k\|,
		\end{align*}
		where the third equality is by the definition of $v^k$. The proof is completed.
	\end{proof}
	\begin{assumption}\label{ass1}
		$\nabla^2\!f$ is locally Lipschitz continuous on $\mathbb{R}^m$.
	\end{assumption}
	
	Assumption \ref{ass1} is a common one in the convergence analysis of Newton-type methods (see, e.g., \cite{Mordu20}). By the Heine-Borel open covering theorem,
	one can show that under Assumption \ref{ass1} the Hessian $\nabla^2\!f$ is Lipschitz continuous on any compact subset of $\mathbb{R}^m$.
	We next use this fact to prove that $\{\alpha_k\}_{k\in\mathcal{K}_2}$ has a uniform lower bound,  which will be employed to establish the sufficient decrease of $\{F(x^k)\}_{k\in\mathbb{N}}$; see \cref{lemma-sequence}.
	\begin{lemma}\label{ls-Newton}
		Under \cref{ass1} there is $\underline{\alpha}>0$ such that for all $k\in\mathcal{K}_2$,
		$\alpha_k\ge\underline{\alpha}$.
	\end{lemma}
	\begin{proof}
		Let $C\!:=\mathcal{L}_{F}(x^0)+\frac{1}{2}\nu\mathbb{B}$. It is easy to check that $A(C)$ is a compact subset of $\mathbb{R}^m$. By invoking \cref{ass1}, there exists a constant $L_{\!\nabla^2\!f}>0$ such that
		\begin{equation}\label{Hessf-ineq}
			\|\nabla^2\!f(Ay)-\nabla^2\!f(Az)\|\le L_{\!\nabla^2\!f}\|A(y-z)\|
			\quad\forall y,z\in C.
		\end{equation}
		Fix any integer $m\ge 0$ with $\beta^m\le\min\left\{1,\frac{1}{2}\nu b_2 r_{\!\rm max}^{\sigma-1}\right\}$, where $\nu$ is the same as the one in \cref{prop-xbark}.
		Fix any $k\in\!\mathcal{K}_2$. From $d_{S_k^{c}}^k=0$, $|x_i^k|>\nu$ for all $i\in S_k$ (\cref{prop-xbark}(i)) and \cref{lemma-bound} (i), we have ${\rm sign}(x^k\!+\tau\beta^md^k)={\rm sign}(x^k)$
		and $|x^k\!+\!\tau\beta^md^k|_{\min}>\frac{\nu}{2}$ for all $\tau\in[0,1]$.
		By Lemma \ref{property-FS} (i), $F_{S_k}$ is twice continuously differentiable on an open set containing the line segment between $u^k$ and $u^k\!+\!\beta^md_{S_k}^k$. From the mean-value theorem,  
		\begin{align}\label{Taylor-FSk}
			&F_{S_k}(u^k\!+\!\beta^md_{S^k}^k)-F_{S_k}(u^k)-\langle r^k,\beta^md_{S_k}^k\rangle\nonumber\\
			&=\frac{1}{2}\beta^{2m}\langle \nabla^2F_{S_k}(u^k\!+\!\tau_k \beta^md^k_{S_k})d_{S_k}^k, d_{S_k}^k\rangle
			\ \ {\rm for\ some}\ \tau_k\in[0,1].
		\end{align}
		Note that $x^k\!+\tau\beta^md^k\in C$ for all $\tau\in[0,1]$ by \cref{lemma-bound} (i). By using \cref{property-FS} (ii) with $\kappa=\nu/2$,
		there exists a constant $\widehat{c}_3>0$ (independent of $k$) such that
		\begin{align*}
			\|\nabla^2\!g_{S_k}(u^k)\!-\!\nabla^2\!g_{S_k}(u^k\!+\!\tau_k\beta^md_{S_k}^k)\|
			& \le\!\int_{0}^{\tau_k}\!\|D^3g_{S_k}(u^k\!+\!t\beta^md_{S_k}^k)\beta^md_{S_k}^k\|dt \\
			& \le \tau_k\widehat{c}_3\beta^m\|d^k_{S_k}\|.
		\end{align*}
		In addition, since $x^k, x^k+\tau_k\beta^m d^k\in C$, using inequality \cref{Hessf-ineq}
		with $y=x^k$ and $z=x^k\!+\!\tau_k\beta^md^k$ and noting that
		${\rm supp}(x^k)={\rm supp}(x^k\!+\tau_k\beta^md^k)=S_k$, we have
		\[
		\|A_{S_k}^{\mathbb{T}}\nabla^2\!f(A_{S_k}u^k)A_{S_k}\!-\!A_{S_k}^{\mathbb{T}}\nabla^2\!f(A_{S_k}(u^k\!+\!\tau_k\beta^md_{S_k}^k))A_{S_k}\|
		\le \tau_kL_{\!\nabla^2\!f}\|A_{S_k}\|^3\beta^m\|d^k_{S_k}\|.
		\]
		From the last two inequalities with the expression of $\nabla^2F_{S_k}$, it follows that
		\begin{equation}\label{Hessian-Lip}
			\|\nabla^2F_{S_k}(u^k)\!-\!\nabla^2F_{S_k}(u^k\!+\!\tau_k\beta^md_{S_k}^k)\|
			\le (L_{\!\nabla^2\!f}\|A_{S_k}\|^3\!+\!\lambda\widehat{c}_3)\beta^m\|d^k_{S_k}\|.
		\end{equation}
		Combining \cref{Taylor-FSk}-\cref{Hessian-Lip} with (S4) of \cref{hybrid} and recalling that $H^k=\nabla^2F_{S_k}(u^k)$, we obtain
		\begin{align}
			&\ F_{S_k}(u^k) - F_{S_k}(u^k\!+\!\beta^md^k_{S_k}) + \varrho\beta^m\langle r^k, d^k_{S_k}\rangle \nonumber\\
			&=(1\!-\!\varrho)\beta^m\left\langle(H^k\!+\!b_1\zeta_k I\!+\!b_2\| r^k\|^\sigma I)d^k_{S_k},d^k_{S_k}\right\rangle\nonumber\\
			&\quad\ \!-\!\frac{1}{2}\beta^{2m}\langle d^k_{S_k},\nabla^2\!F_{S_k}(u^k\!+\!\tau_k\beta^md^k_{S_k})d^k_{S_k}\rangle \nonumber\\
			&\ge\frac{1}{2}b_2\beta^m \|r^k\|^\sigma \|d_{S_k}^k\|^2 +\frac{1}{2}\beta^{2m}\!\left\langle
			\left(H^k\!-\!\nabla^2F_{S_k}(u^k\!+\!\tau_k\beta^md^k_{S_k})\right)d^k_{S_k}, d^k_{S_k}\right\rangle \nonumber\\
			&\ge \frac{1}{2}b_2\beta^m \|r^k\|^\sigma \|d_{S_k}^k\|^2
			-\frac{1}{2}(L_{\!\nabla^2\!f}\|A_{S_k}\|^3\!+\!\lambda\widehat{c}_3)\beta^{3m}\|d^k_{S_k}\|^3 \nonumber \\
			&=\frac{1}{2}\beta^m\|d_{S_k}^k\|^3 \Big( b_2\frac{\|r^k\|^{\sigma}}{\|d^k_{S_k}\|}
			-\widetilde{c}_3\beta^{2m}\Big)\ \ {\rm with}\ \widetilde{c}_3:=L_{\!\nabla^2\!f}\|A_{S_k}\|^3\!+\!\lambda\widehat{c}_3,
		\end{align}
		where the first equality is using $r^k=-G^kd_{S_k}^k$ by (S4), and
		the first inequality is due to $H^k+b_1\zeta_k I \succeq 0,\varrho\in(0,\frac{1}{2}]$
		and $\zeta_k\geq 0$. By the definition of $d_{S_k}^k$ and \cref{lemma-bound} (i),
		\begin{equation}\label{dk-bound}
			\frac{\|d_{S_k}^k\|}{\|r^k\|^{\sigma}} \le \frac{\|(G^k)^{-1}\|\|r^k\|}{\|r^k\|^{\sigma}}
			\leq  \frac{\|r^k\|^{1-2\sigma}}{b_2} \leq \frac{r_{\!\rm max}^{1-2\sigma}}{b_2}.
		\end{equation}
		The above arguments demonstrate that whenever
		$\beta^m\le\!\min\big\{1,\frac{1}{2}\nu b_2r_{\!\rm max}^{\sigma-1},\frac{b_2}{\sqrt{\widetilde{c}_3r_{\!\rm max}^{1-2\sigma}}}\big\}$,
		\[
		F_{S_k}(u^k)\!-F_{S_k}(u^k+\beta^md^k_{S_k})+\varrho\beta^m\langle r^k, d^k_{S_k}\rangle\ge0.
		\]
		Let $\underline{\alpha}\!:=\beta\min\big\{1,
		\frac{1}{2}\nu b_2r_{\!\rm max}^{\sigma-1},\frac{b_2}{\sqrt{\widetilde{c}_3r_{\!\rm max}^{1-2\sigma}}}\big\}$. Then, for all $k\in \mathcal{K}_2$,
		$\alpha_k\ge\underline{\alpha}$.
	\end{proof}
	\subsection{Convergence rate of objective function value sequence}\label{sec4.1}
	We have achieved the convergence of the sequence  $\{F(x^k)\}_{k\in\mathbb{N}}$ in \cref{well-defineness} (i). To establish its convergence rate, we need two technical lemmas. Among others, \cref{lemma-sequence} states that 
	$\{F(x^k)\}_{k\in\mathbb{N}}$ is sufficiently decreasing under \cref{ass1}, while \cref{lemma-gap}
	reveals that under \cref{ass1} the subsequence $\{d^k\}_{k\in\mathcal{K}_2}$ converges to $0$.
	\begin{lemma}\label{lemma-sequence}
		Let $\{x^k\}_{k\in\mathbb{N}}$ and $\{\overline{x}^k\}_{k\in\mathbb{N}}$
		be the sequences yielded by \cref{hybrid}. Then, under \cref{ass1},
		the following assertions hold.
		\begin{itemize}
			\item [(i)] There exists $\widehat{\gamma}>0$ such that for all $k\in\mathbb{N}$,
			$F(x^{k+1})\le F(x^k)-\frac{\widehat{\gamma}}{2}\|x^k -\overline{x}^k \|^2$.
			
			\item[(ii)]  $\lim_{k\rightarrow\infty}\|x^k-\overline{x}^k\|=0$.
			
			\item[(iii)] Every element of $\Omega(x^0)$ is an $L$-type stationary point of \eqref{model}.		
		\end{itemize}
	\end{lemma}
	\begin{proof}
		{\bf(i)-(ii)} By \cref{well-defineness} (i), $\{x^k\}_{k\in\mathbb{N}}$ is contained in the compact set $\mathcal{L}_{F}(x^0)$, while $|x^k|_{\min}>\nu$ for all $k\in\mathcal{K}_2$ by \cref{prop-xbark} (i). Then, by 
		invoking \cref{property-FS} (ii) with $C=\mathcal{L}_{F}(x^0)$ and $\kappa=\nu$, there exists $\widehat{c}_2>0$ (independent of $k$) such that
		$\|H^k\|=\|\nabla^2F_{S_k}(u^k)\|\le\widehat{c}_2$ for all $k\in\mathcal{K}_2$.
		Together with the expression of $G^k$ in (S4) and \cref{lemma-bound} (i), for all $k\in\mathcal{K}_2$,
		\begin{equation}\label{Gk-snorm}
			G^k\preceq [(1\!+\!b_1)\|H^k\|+b_2\|r^k\|^{\sigma}]I\preceq [(1\!+\!b_1)\widehat{c}_2+ b_2r_{\!\rm max}^{\sigma}]I.
		\end{equation}
		From the line search step in (\ref{ls-NT}), \cref{lemma-bound} (ii) and \cref{ls-Newton}, for all $k\in\mathcal{K}_2$,
		\begin{align}\label{curve-ineq}
			& F(x^{k+1})-F(x^k)\le\alpha_k\varrho\langle r^k, d^k_{S_k}\rangle
			= -\alpha_k\varrho\langle r^k,(G^k)^{-1}r^k \rangle\nonumber\\
			&\le -\frac{\varrho\underline{\alpha}}{(1\!+\!b_1)\widehat{c}_2+ b_2r_{\!\rm max}^{\sigma}}\|r^k\|^2
			\le -\frac{\varrho\underline{\alpha}\varpi^2}{16[(1\!+\!b_1)\widehat{c}_2+ b_2r_{\!\rm max}^{\sigma}]}\|x^k-\overline{x}^k\|^2,
		\end{align}
		where the last inequality is using ${\rm sign}(x^k)={\rm sign}(\overline{x}^k)$
		implied by $k\in \mathcal{K}_2$. In addition, by (S1) of \cref{hybrid} and \cref{gist}, 
		$F(x^{k+1})\le F(x^k)-\frac{\widetilde{\alpha}}{2}\|x^k -\overline{x}^k \|^2$ for all $k\in\mathcal{K}_1$.
		Along with the last inequality, part (i) holds with
		$\widehat{\gamma}=\!\min\big\{\frac{\varrho\underline{\alpha}\varpi^2}{8[(1\!+\!b_1)\widehat{c}_2+ b_2r_{\!\rm max}^{\sigma}]},\widetilde{\alpha}\big\}$. From part (i) and the convergence of $\{F(x^k)\}_{k\in\mathbb{N}}$, we obtain part (ii).
		
		\noindent
		{\bf(iii)} Pick any $x^*\in\Omega(x^0)$. There exists a subsequence $\{x^{k_j}\}_{j\in\mathbb{N}}$
		such that $x^{k_j}\to x^*$ as $j\to\infty$. From part (ii),
		$\lim_{j\to\infty}\overline{x}^{k_j}=x^*$. For each $j\in\mathbb{N}$, from (S1)
		we have $\overline{x}^{k_j}\in{\rm prox}_{\overline{\mu}_{k_j}^{-1}(\lambda g)}
		(x^{k_j}\!-\!\overline{\mu}_{k_j}^{-1}\nabla\psi(x^{k_j}))$; while by \cref{well-defineness} (iii),  $\overline{\mu}_{k_j}\in[\mu_{\rm min},\widetilde{L})$. We assume that $\overline{\mu}_{k_j}\to\overline{\mu}_*$ (if necessary taking a subsequence). Define the mapping $\mathcal{F}(\mu,x):={\rm prox}_{\mu^{-1}(\lambda g)}(x\!-\!\mu^{-1}\nabla\psi(x))$
		for $x\in\mathbb{R}^n$ and $\mu\in[\mu_{\rm min},\widetilde{L})$.
		By \cite[Proposition 4.4]{BS00}, the mapping $\mathcal{F}$ is upper semicontinuous, so it is outer semicontinuous at $(\overline{\mu}_*,x^*)$  by \cite[p. 138-139]{pang03}.
		Thus, $x^*\in{\rm prox}_{\overline{\mu}_*^{-1}(\lambda g)}
		(x^*\!-\!\overline{\mu}_{*}^{-1}\nabla\psi(x^*))$,
		and the result follows.
	\end{proof}
	\begin{lemma}\label{lemma-gap}
		Let $\{x^k\}_{k\in\mathbb{N}}$ and $\{\overline{x}^k\}_{k\in\mathbb{N}}$ be the sequences given by \cref{hybrid}. Then, under \cref{ass1} there exists a constant $\widehat{c}_2>0$ such that for all $k\in \mathcal{K}_2$,
		$$
		{\rm dist}(0,\partial F(x^k))
		\le\!\widetilde{c}_2\|x^{k}\!-\!\overline{x}^{k}\|
		\ \ {\rm with}\
		\widetilde{c}_2=\widehat{L}+\widetilde{L}+\widehat{c}_2,
		$$
		and consequently, $\lim_{\mathcal{K}_2 \ni k\to\infty}\|r^k\|=0$ and
		$ \lim_{\mathcal{K}_2\ni k\rightarrow \infty}\|d^k\|=0$.
	\end{lemma}
	\begin{proof}
		Fix any $k\in\mathcal{K}_2$. Since $\overline{x}^{k}\in{\rm prox}_{\overline{\mu}_{k}^{-1}(\lambda g)}\big(x^{k}\!-\!\overline{\mu}_{k}^{-1}\nabla\psi(x^k)\big)$, by \cite[Exercise 8.8]{RW09}, we have
		$0\in \nabla\psi(x^k)+\overline{\mu}_{k}(\overline{x}^{k}-x^{k})+\lambda\partial g(\overline{x}^{k})$, which implies that 
		\[
		\nabla\psi(\overline{x}^k)-\nabla\psi(x^k)
		+\overline{\mu}_{k}(x^{k}-\overline{x}^{k})\in\partial F(\overline{x}^{k}).
		\]
		Recall that $\nabla\psi$ is Lipschitz continuous on the compact set $\mathcal{L}_{F}(x^0)$ with Lipschitz constant not more than $\widehat{L}$,
		which is the same as the one appearing in the proof of \cref{well-defineness} (iii). Then,  $\|\nabla\psi(x^k)-\nabla\psi(\overline{x}^k)\| \leq \widehat{L}\|x^k-\overline{x}^k\|$.
		Together with the last inclusion, using
		$\overline{\mu}_k<\widetilde{L}$ by \cref{well-defineness} (iii) yields that
		\begin{equation}\label{xbarbound}
			\|\nabla F_{S_k}(\overline{u}^{k})\|
			={\rm dist}(0,\partial F(\overline{x}^k))
			\leq (\widehat{L}+\widetilde{L})\|x^k-\overline{x}^k\|.
		\end{equation}
		Let $C$ be a bounded open convex set containing $\mathcal{L}_{F}(x^0)$. By \cref{well-defineness} (i) and the convexity of $C$, $\overline{u}^k+\tau(u^k\!-\!\overline{u}^k)\in C$ for all $\tau\in[0,1]$. Recall that $k\in\mathcal{K}_2$. Hence, $x^k \neq 0$ and ${\rm sign}(u^k)={\rm sign}(\overline{u}^k)$.
		Together with Lemma \ref{prop-xbark} (i), for all $\tau\in[0,1]$, we have $\left|\overline{u}^k+\tau(u^k\!-\!\overline{u}^k)\right|_{\min}\ge\nu$
		and ${\rm sign}(\overline{u}^k\!+\!\tau(u^k\!-\!\overline{u}^k))={\rm sign}(u^k)$. By \cref{property-FS} (ii), there exists a constant $\widehat{c}_2>0$ (independent of $k$) such that for all $\tau\in[0,1]$, $\|\nabla^2F_{S_k}(\overline{u}^k\!+\!\tau(u^k\!-\!\overline{u}^k))\|
		\le\widehat{c}_2$. Note that ${\rm dist}(0,\partial F(x^k))=\|r^k\|$ and
		\begin{align}
			\|r^k\|& =\|r^k\!-\!\nabla F_{S_k}(\overline{u}^{k})+\nabla F_{S_k}(\overline{u}^{k})\|\le \|r^k\!-\!\nabla F_{S_k}(\overline{u}^k)\|+\|\nabla F_{S_k}(\overline{u}^{k})\|\nonumber\\
			&\le\int_{0}^{1}\|\nabla^2F_{S_k}(\overline{u}^k\!+\!\tau(u^k\!-\!\overline{u}^k))(u^k\!-\!\overline{u}^k)\|d\tau\nonumber
			+\|\nabla F_{S_k}(\overline{u}^{k})\| \nonumber\\
			&\leq \widehat{c}_2\|u^k\!-\!\overline{u}^k\|+\|\nabla F_{S_k}(\overline{u}^{k})\|
			\le\big[(\widehat{L}+\widetilde{L})+\widehat{c}_2\big]
			\|x^{k}-\overline{x}^{k}\|, \label{rk-bound}
		\end{align}
		where the last inequality is due to \eqref{xbarbound}. The first part of the conclusions	follows. From 
		\cref{rk-bound}, \cref{lemma-sequence} (ii) and \cref{dir-bound}, we obtain the second part.
	\end{proof}
	
	To achieve the linear convergence rate of the objective sequence $\{F(x^k)\}_{k\in\mathbb{N}}$,
	we first argue that for all sufficiently large $k$, the support of the iterate $x^k$ is stable, and $k\in\mathcal{K}_2$. The latter means that after a finite number of iterates,
	\cref{hybrid} reduces to a regularized Newton method for minimizing the function $F_{S_*}$, where $S_*$ is defined below in \cref{supp-lemma} (i).

	\begin{lemma}\label{supp-lemma}
		Let $\{x^k\}_{k\in\mathbb{N}}$ and $\{\overline{x}^k\}_{k\in\mathbb{N}}$
		be the sequences given by \cref{hybrid}. Then, under \cref{ass1},
		the following assertions hold.
		\begin{itemize}
			\item[(i)] There exists an index set $S_{*}\subseteq[n]$ such that for all sufficiently large $k$,
			\[
			{\rm supp}(x^k) = {\rm supp}(\overline{x}^k) =S_{*};
			\]
			furthermore, every cluster point $x^*$ of $\{x^k\}_{k\in\mathbb{N}}$ satisfies
			${\rm supp}(x^*)=S_{*}$.
			
			\item[(ii)] There exists $\widehat{k}\in\mathbb{N}$ such that for all $k\ge\widehat{k}$, $k\in\mathcal{K}_2$.
		\end{itemize}
	\end{lemma}
	\begin{proof}
		{\bf(i)} First we argue that  $|x^k|_{\min}>\frac{\nu}{2}$  for all sufficiently large $k$. Indeed, by \cref{prop-xbark} (i), if $k-\!1\in \mathcal{K}_1$,
		i.e., $x^k=\overline{x}^{k-1}$, we have $|x^k|_{\min}\ge\nu$.
		If $k-1\in\mathcal{K}_2$, we have $|x^{k-1}|_{\min}\ge\nu$, while 
		by \cref{lemma-gap}, for all sufficiently large $k$, $\|d^{k-1}\|<\frac{\nu}{3}$, which along with $x^{k}=x^{k-1}\!+\alpha_kd^k,\alpha_k\in(0,1]$,
		$d_{S_k^{c}}^k=0$ and $|x^{k-1}|_{\min}\ge\nu$ implies that $|x^k|_{\min}>\frac{\nu}{2}$.
		Next we argue that for all sufficiently large $k$, ${\rm supp}(x^k)={\rm supp}(\overline{x}^k)$. Indeed, by \cref{lemma-sequence} (ii),
		for all sufficiently large $k$,
		$\|x^k-\overline{x}^{k}\|<\frac{\nu}{3}$. Hence, for every $i\in{\rm supp}(x^k)$, we have
		$|\overline{x}^k_i|\geq|x_i^k|-|x_i^k-\overline{x}_i^{k}|>\frac{\nu}{2}-\frac{\nu}{3} > 0$, which implies that ${\rm supp}(x^k) \subseteq {\rm supp}(\overline{x}^k)$; and for every $i\in {\rm supp}(\overline{x}^k)$, we have $|x^k_i| > |\overline{x}^k_i|-\frac{\nu}{3} > 0$,
		which implies that ${\rm supp}(\overline{x}^k) \subseteq {\rm supp}(x^k)$.
		Thus, ${\rm supp}(\overline{x}^k)={\rm supp}(x^k)$ holds for all $k$ large enough. It remains to show that for all $k$ large enough, ${\rm supp}(x^k) = {\rm supp}(x^{k+1})$. For all sufficiently large $k\in \mathcal{K}_1$, the conclusion holds since $x^{k+1}=\overline{x}^k$ and ${\rm supp}(\overline{x}^k) = {\rm supp}(x^k)$. For all sufficiently large $k\in \mathcal{K}_2$, by \cref{lemma-gap}, we have $\|d^k\|<\frac{\nu}{3}$ and then $\|x^{k+1}\!-x^k \|<\frac{\nu}{3}$, and the conclusion follows by the above arguments.
		To sum up, ${\rm supp}(x^{k+1})={\rm supp}(x^k)= {\rm supp}(\overline{x}^k)$ holds for all sufficiently large $k$. Since  $|x^k|_{\min} > \frac{\nu}{2}$ for all sufficiently large $k$,  following a similar arguments as above we have every cluster point $x^*$ of $\{x^k\}$ satisfies ${\rm supp}(x^*)=S_{*}$.
		
		\noindent
		{\bf (ii)}  By the proof of part (i), we have $|x^k|_{\min} > \frac{\nu}{2}$ for all sufficiently large $k$. Together with \cref{lemma-sequence} (ii) and \cref{prop-xbark} (iii),  the two conditions in \eqref{if-else}
		are satisfied for all $k$ large enough, so there exists $\widehat{k}\in\mathbb{N}$ such that for all $k\ge\widehat{k}$,
		$k\in\mathcal{K}_2$.
	\end{proof}
	
	Now we are in a position to achieve the $Q$-linear convergence rate of the objective value
	sequence $\{F(x^k)\}_{k\in\mathbb{N}}$ under the KL property of the exponent $1/2$ of $F$.
	\begin{proposition}\label{obj-rlinear}
		Suppose that \cref{ass1} holds, and that $F$ is a KL function of
		exponent $1/2$. Then $\{F(x^k)\}_{k\in\mathbb{N}}$ converges to some value $F^*$
		in a $Q$-linear rate. 
	\end{proposition}
	\begin{proof}
		If there exists some $k\in\mathbb{N}$ such that $F(x^{k})=F(x^{k+1})$, by \cref{lemma-sequence} (i), we have $x^{k}=\overline{x}^{k}$, and the stopping condition in (S2) is satisfied, so $\{x^k\}_{k\in\mathbb{N}}$ converges to an $L$-type stationary point within a finite number of steps.
		Hence, it suffices to consider that $F(x^k)>F(x^{k+1})$ for all $k\in\mathbb{N}$.
		Since $F$ is assumed to be a KL function of exponent $1/2$,
		by \cite[Lemma 6]{Bolte14}  and \cref{well-defineness} (ii), there exist $\varepsilon>0$ and $\eta>0$
		such that for all $\overline{x}\in\Omega(x^0)$
		and all $z\in\{x\in\mathbb{R}^n\,|\,{\rm dist}(x,\Omega(x^0))<\varepsilon\}\cap[F(\overline{x})<F<F(\overline{x})+\eta]$,
		\begin{equation}\label{KL-tempineq}
			\varphi'(F(z) -F(\overline{x})){\rm dist}(0,\partial F(z)) \ge 1,
		\end{equation}
		where $\varphi(t)=c\sqrt{t}$ for some $c>0$. Let $x^*$ be a cluster point of $\{x^k\}_{k\in\mathbb{N}}$.
		Clearly, $\lim_{k\rightarrow\infty} {\rm dist}(x^k,\Omega(x^0))=0$.
		Along with $\lim_{k\to\infty}F(x^k)=F(x^*)$, for all sufficiently large $k$,
		$x^k\in\{x\in\mathbb{R}^n\,|\,{\rm dist}(x,\Omega(x^0))<\varepsilon\}\cap[F(x^*)<F<F(x^*)+\eta]$,
		and then
		\[
		\frac{c}{2} (F(x^k) - F(x^*))^{-1/2} {\rm dist} (0,\partial F(x^k)) \geq 1.
		\]
		Let $\Delta_k=F(x^k)-F(x^*)$ for each $k$. By \cref{supp-lemma} (ii), when $k>\widehat{k}$, $k\in\mathcal{K}_2$.
		Combining the above inequality with \cref{lemma-gap} yields that for all $k>\widehat{k}$ (if necessary by increasing $\widehat{k}$),
		\begin{align*}
			4c^{-2}&\le \big[(\Delta_k)^{-1/2} {\rm dist} (0,\partial F(x^k))\big]^2
			\le \widetilde{c}_2^2(\Delta_k)^{-1}\|x^k-\overline{x}^k\|^2 \\
			&\le 2\widetilde{c}_2^2\widehat{\gamma}^{-1}(\Delta_k)^{-1}[F(x^k) - F(x^{k+1})]
			= 2\widetilde{c}_2^2\widehat{\gamma}^{-1}(\Delta_k)^{-1}(\Delta_k-\Delta_{k+1}),
		\end{align*}
		where the third inequality is due to \cref{lemma-sequence} (i). 
		The last inequality, along with $0<\Delta_{k+1}<\Delta_k$ implies that $\rho\!=\!1-\frac{2\widehat{\gamma}}{(c\widetilde{c}_2)^2} \in (0,1)$. Then, for all  $k\ge\widehat{k}$, we have $\Delta_{k+1}\le\rho\Delta_k$, so that $\{F(x^k)\}_{k\in\mathbb{N}}$ converges to $F^*=F(x^*)$ in a $Q$-linear rate.
	\end{proof}
	\subsection{Convergence analysis of iterate sequence}\label{sec4.3}

	In order to achieve the convergence of the sequence $\{x^k\}_{k\in\mathbb{N}}$,
	we also need the following assumption:
	\begin{assumption}\label{ass2}
		It holds that ${\displaystyle\liminf_{\mathcal{K}_{2}\ni k\rightarrow\infty}}\frac{-\langle r^k, d_{S_k}^k\rangle}{\|r^k\| \|d^k_{S_k}\|}>0$.
	\end{assumption}
	
	Assumption \ref{ass2} is very common in the global convergence analysis of line search Newton-type methods (see, e.g., \cite{Wright02}), which essentially requires that the angle between $r^k$
	and $d_{S_k}^k$ is sufficiently away from $\pi/2$ and close to $\pi$. Note that the early global convergence analysis of Newton-type methods aims to  achieve $\lim_{k\rightarrow\infty} \|r^k\| = 0$
	under Assumption \ref{ass2}. Here, under this assumption, we establish the convergence of the whole iterate sequence for the KL function $F$. 
	\begin{theorem}\label{gconverge}
		Suppose Assumptions \ref{ass1} and \ref{ass2} hold. The following assertions hold.
		\begin{itemize}
			\item [(i)] If $F$ is a KL function, then $\sum_{k=1}^\infty \|x^{k+1}\!-\!x^k\|<\infty$,
			and consequently, $\{x^k\}_{k\in\mathbb{N}}$ converges to an $L$-type stationary
			point of \eqref{model}, say $x^*$.
			
			\item[(ii)] If $F$ is a KL function of exponent $1/2$ at $x^*$, then $\{x^k\}_{k\in\mathbb{N}}$ converges $R$-linearly to $x^*$.
		\end{itemize}
	\end{theorem}
	\begin{proof}
		{\bf(i)} By the proof of \cref{obj-rlinear},
		it suffices to consider the case where $F(x^k)>F(x^{k+1})$ for all $k$.
		Let $x^*$ be a cluster point of $\{x^k\}_{k\in\mathbb{N}}$.
		Following a similar argument to the proof of \cref{obj-rlinear}, we have for sufficiently large $k$,
		\begin{equation}\label{KLexpre}
			\varphi'(F(x^k) -F(x^*)) {\rm dist}(0, \partial F(x^k)) \geq 1.
		\end{equation}
		By \cref{ass2}, there exists $c_{\rm min}>0$ such that
		for all sufficiently large $k \in \mathcal{K}_2$,
		\begin{equation}\label{ass2-coro}
			-\langle r^k, d^k_{S_k}\rangle >c_{\rm min}\|r^k\|\| d^k_{S_k}\|.
		\end{equation}
		By \cref{supp-lemma},  there exists $\widehat{k}\in\mathbb{N}$ such that for all $k\ge\widehat{k}$, $k\in\mathcal{K}_2$ and $S_k = S_{k+1}$. Together with  \eqref{ls-NT} and \eqref{ass2-coro}, if ncessary by increasing $\widehat{k}$, for all $k\ge\widehat{k}$, we have
		\begin{equation}\label{FS-nabla}
			\frac{F(x^{k}) - F(x^{k+1})}{\|r^k\|}\ge \frac{-\varrho \alpha_k \langle r^k, d_{S_k}^k\rangle }{\|r^k\|}
			\ge \varrho c_{\rm min}\|\alpha_k d^k_{S_k}\|=\varrho c_{\rm min}\|x^{k+1}-x^k\|.
		\end{equation}
		In addition, from the concavity of $\varphi$ on $[0,\eta)$, for all $k>\widehat{k}$, it holds that
		\begin{equation}\label{KL-concave}
			\varphi(F(x^k)\!-\!F(x^*))-\varphi(F(x^{k+1})\!-\!F(x^*))
			\geq \varphi'(F(x^k)\!-\!F(x^*))(F(x^k)\!-\!F(x^{k+1})).
		\end{equation}
		For each $k$, let $\bar{\Delta}_{k}:=\varphi(F(x^k)\!-\!F(x^*))$.
		From \cref{KLexpre} and \cref{FS-nabla}-\eqref{KL-concave}, if possibly enlarging $\widehat{k}$, we have
		for all $k\ge\widehat{k}$,
		\begin{align*}
			\bar{\Delta}_k-\bar{\Delta}_{k+1}&\ge \varphi'(F(x^k)-F(x^*))(F(x^{k}) - F(x^{k+1})) \\
			&\ge\frac{F(x^k)-F(x^{k+1})}{{\rm dist}(0,\partial F(x^k))}
			= \frac{F(x^k)-F(x^{k+1})}{\|r^k\|} \geq \varrho c_{\rm min}\|x^{k+1}-x^k\|.
		\end{align*}
		Summing this inequality from $\widehat{k}$ to any $k>\widehat{k}$ yields that
		\begin{equation*}
			\sum_{j=\widehat{k}}^k\|x^{j+1}\!-\!x^{j}\|
			\le\frac{1}{\varrho c_{\rm min}} \sum_{j=\widehat{k}}^k(\bar{\Delta}_j\!-\!\bar{\Delta}_{j+1})
			= \frac{1}{\varrho c_{\rm min}}(\bar{\Delta}_{\widehat{k}}\!-\!\bar{\Delta}_{k+1})
			\leq \frac{1}{\varrho c_{\rm min}}\bar{\Delta}_{\widehat{k}}.
		\end{equation*}
		Passing the limit $k\to\infty$ to this inequality yields that
		$\sum_{j=\widehat{k}}^\infty \|x^{j+1}\!-\!x^j\|<\infty$. Thus the sequence $\{x^k\}$ converges. By Lemma \ref{lemma-sequence} (iii), the desired result then follows.
		
		\noindent
		{\bf(ii)} 
		For each $k\in\mathbb{N}$, write $\Delta_k:=F(x^k)-F(x^*)$.
		From \cref{supp-lemma} and the proof of \cref{obj-rlinear}, there exists
		$\widehat{k}$ such that for all $k>\widehat{k}$, $k\in \mathcal{K}_2$ and $\Delta_{k+1}\le\rho\Delta_k$. From this recursion formula, 
		\begin{equation}\label{r-linear}
			F(x^k) - F(x^*)\le\Delta_{\widehat{k}}\rho^{k-\widehat{k}}.
		\end{equation}
		By \cref{dir-bound} and \cref{lemma-gap}, for all $k\ge\widehat{k}$,
		$\|d^k \|\le b_2^{-1}\widetilde{c}_2^{1-\sigma}\|x^k-\overline{x}^k\|^{1-\sigma}$.
		Together with part (i), \cref{lemma-sequence} (i) and \eqref{r-linear},
		for all $k\ge\widehat{k}$ it holds that
		\begin{align*}
			\|x^k - x^*\|&\le \sum_{j=k}^\infty \|x^j-x^{j+1}\|
			=\sum_{j=k}^\infty\alpha_j\|d^j\|
			\le\sum_{j=k}^\infty\|d^j\|\le b_2^{-1}\widetilde{c}_2^{1-\sigma}\sum_{j=k}^\infty\|x^j-\overline{x}^j\|^{1-\sigma} \\
			&\le b_2^{-1}\widetilde{c}_2^{1-\sigma}\sum_{j=k}^\infty\Big(\frac{2(F(x^j)-F(x^{j+1}))}
			{\widehat{\gamma}}\Big)^{\frac{1-\sigma}{2}}\\
			&\le b_2^{-1}\widetilde{c}_2^{1-\sigma}\Big(\frac{2\Delta_{\widehat{k}}}
			{\widehat{\gamma}\rho^{\widehat{k}}}\Big)^{\frac{1-\sigma}{2}}\sum_{j=k}^\infty\rho^{\frac{(1-\sigma)j}{2}}
			\le\Big(\frac{2\Delta_{\widehat{k}}}
			{\widehat{\gamma}\rho^{\widehat{k}}}\Big)^{\frac{1-\sigma}{2}}\!
			\frac{\widetilde{c}_2^{1-\sigma}}{b_2(1\!-\!\rho^{1/4})}\rho^{k/4}.
		\end{align*}
		This means that the sequence $\{x^k\}_{k\in\mathbb{N}}$ converges to $x^*$ in an $R$-linear rate.
	\end{proof}
	
	By \cref{prop-FS-KL}, to check the KL property with exponent $1/2$ of $F$ at $x^*$, it suffices to verify that of $F_{S_*}$ at $x^*_{S_*}$, and due to the sufficient smoothness of $F_{S_*}$ at $x^*_{S_*}$, the verification of the latter is easier than that of the former. By  \cite[Lemma 3]{Xu15}, the nonsingularity of $\nabla^2 \!F_{S_*}(x^*_{S_*})$ implies the KL property of exponent $1/2$  for $F_{S_*}$ at $x^*_{S_*}$.
	
	\medskip
	By \cref{gconverge}, if Assumptions \ref{ass1}-\ref{ass2} hold and $F$ is a KL function, the sequence $\{x^k\}_{k\in\mathbb{N}}$ is convergent. In the sequel, we denote its limit by $x^*$. By \cref{supp-lemma} (i), ${\rm supp}(x^*) = S_*$. Write
	$$ u^* := x^*_{S_*} \ {\rm and} \ \ \mathcal{U}^*\!:=\!\big\{u\in\mathbb{R}^{|S_{*}|}\,|\,\nabla\!F_{\!S_{*}}(u)=0,\nabla^2\!F_{\!S_{*}}(u)\succeq 0\big\}.$$ 
	To achieve the superlinear convergence rate of $\{x^k\}_{k\in\mathbb{N}}$, we need to bound $\zeta_k$ involved in the matrix $G_k$ by ${\rm dist}(u^k,\mathcal{U}^*)$ as in the following lemma.
	\begin{lemma}\label{lemma-zetak}
		Suppose that Assumptions \ref{ass1} and \ref{ass2} hold, and that $F$ is a KL function.
		If $\nabla^2F_{\!S_{*}}(u^*)\succeq 0$, then there exists $c_H>0$ such that for all sufficiently large $k$, $\zeta_k\le c_H{\rm dist}(u^k,\mathcal{U}^*)$.
	\end{lemma}
	\begin{proof}
		By the proof of \cref{supp-lemma} (i), we have $|x^*|_{\min} \geq \frac{\nu}{2}$. Fix any $\varepsilon<\frac{\nu}{4}$. From \cref{supp-lemma} and \cref{gconverge} (i), there exists $\widetilde{k}\in\mathbb{N}$ such that for all $k>\widetilde{k}$, $k\in\mathcal{K}_2$, $S_k = S_*$ and $u^k \in \mathbb{B}(u^*, \varepsilon/2)$. By following the proof of \cref{ls-Newton}, there exists $c_H>0$ such that for any $u', u''\in \mathbb{B}(u^*, \varepsilon)$, 
		\begin{equation}\label{Hessian-Lip1}
			\|\nabla^2\!F_{\!S_*}(u') - \nabla^2\!F_{\!S_*}(u'')\|\leq c_H\|u' - u''\|.
		\end{equation}
		Fix any $k>\widetilde{k}$.   
		When $\lambda_{\min}(\nabla^2\!F_{\!S_k}(u^k)) > 0$, the desired result is trivial, so it suffices to consider the case $\lambda_{\rm min}(\nabla^2\!F_{\!S_k}(u^k)) \leq 0$. 
		Pick any $\widetilde{u}^{k}\in\Pi_{\mathcal{U}^*}(u^k)$. Since $u^* \in \mathcal{U}^*$, one can deduce that $\|\widetilde{u}^k - u^*\| \leq \| \widetilde{u}^{k} - u^k\| + \|u^k - u^*\| \leq 2\|u^k - u^*\| \leq \varepsilon.$
		If $\lambda_{\rm min}(\nabla^2\!F_{\!S_k}(\widetilde{u}^{k}))=0$, 
		then  by Weyl's inequality \cite[Corollary III.2.6]{bhatia13} we have $\zeta_k=-\lambda_{\rm min}(\nabla^2\!F_{\!S_k}(u^k))
		\le \|\nabla^2\!F_{\!S_k}(\widetilde{u}^{k})\!-\!\nabla^2\!F_{\!S_k}(u^k)\|$, 
		which together with \eqref{Hessian-Lip1} implies that 
		$\zeta_k\le c_H\|u^k\!-\!\widetilde{u}^{k}\|=c_H{\rm dist}(u^k,\mathcal{U}^*)$.
		Now suppose that $\lambda_{\rm min}(\nabla^2\!F_{\!S_k}(\widetilde{u}^{k}))>0$.
		Let $\phi_{k}(t):=\lambda_{\rm min}[\nabla^2\!F_{S_k}(u^k\!+\!t(\widetilde{u}^{k}\!-\!u^k))]$
		for $t\ge 0$. Clearly, $\phi_{k}$ is continuous on an open interval containing $[0,1]$.
		Note that $\phi_k(0)<0$ and $\phi_k(1)>0$. There necessarily exists $\overline{t}_k\in(0,1)$
		such that $\phi_k(\overline{t}_k)=0$. Consequently, by Weyl's inequality,
		\begin{align*}
			\zeta_k
			&=\big[\lambda_{\rm min}(\nabla^2\!F_{S_k}(u^k\!+\!\overline{t}_k(\widetilde{u}^{k}\!-\!u^k)))
			-\lambda_{\rm min}(\nabla^2\!F_{S_k}(u^k))\big]\\
			&\le \|\nabla^2\!F_{S_k}(u^k\!+\!\overline{t}_k(\widetilde{u}^{k}\!-\!u^k))\!-\!\nabla^2\!F_{S_k}(u^k)\|
			\le c_H\|\widetilde{u}^{k}\!-\!u^k\|.
		\end{align*}
		This shows that the desired result holds. The proof is completed.
	\end{proof}
	
	Ueda and Yamashita ever obtained a similar result in \cite[Lemma 5.2]{Ueda10} under the condition that $\mathcal{U}^*$ is the set of local minima of $F_{S_*}$. Here, we remove the local optimality of $\mathcal{U}^*$ and provide a simpler proof. Based on this result, 
	we establish the superlinear convergence rate of $\{x^k\}_{k\in \mathbb{N}}$ under a local error bound condition.
	
	\begin{theorem}\label{lemma-newtondir}
		Suppose that Assumptions \ref{ass1} and \ref{ass2} hold, and that $F$ is a KL function.
		If $\nabla^2F_{\!S_{*}}(u^*)\succeq 0$ 
		and there exist $\delta>0$ and $\kappa_1>0$ such that for all $u\in\mathbb{B}(u^*,\delta)$
		\begin{equation}\label{err-bound}
			\kappa_1{\rm dist}(u,\mathcal{U}^*)\!\le\!\|\nabla F_{\!S_{*}}(u)\|,
		\end{equation}
		then the sequence $\{x^k\}_{k\in\mathbb{N}}$ converges to $x^*$ in a Q-superlinear rate of order $1\!+\!\sigma$.
	\end{theorem}
	\begin{proof}
		By \cref{gconverge} and \cref{supp-lemma}, there exists
		$\widetilde{k} \in \mathbb{N}$ such that for all $k\ge\widetilde{k}$, $k\in\mathcal{K}_2$ and $S_k=S_{*}$.
		By comparing the iterate steps of \cref{hybrid} for $k\ge\widetilde{k}$ with those of E-RNM proposed in \cite{Ueda10}, we conclude that the sequence $\{u^k\}_{k\ge\widetilde{k}}$
		is the same as the one generated by E-RNM of \cite{Ueda10}. 
		By \cref{lemma-zetak},
		there exists
		a constant $c_{H}>0$ such that for all $k\ge\widetilde{k}$ (if necessary by increasing $\widetilde{k}$), $\zeta_k\le c_{H}{\rm dist}(u^k,\mathcal{U}^*)$.
		Then, by \cite[Theorem 5.1]{Ueda10}
		${\rm dist}(u^k,\mathcal{U}^*)$ converges to $0$ superlinearly with rate $1\!+\!\sigma$.
		
		Write $\mathcal{X}^*\!:=\!\{x\in\mathbb{R}^n\,|\,x_{S_*}\!\in\mathcal{U}^*,x_{S_{*}^{c}}=0\}$. For all $k\ge\widetilde{k}$, 
		from $S_{k}=S_{*}$, clearly, ${\rm dist}(x^k,\mathcal{X}^*)={\rm dist}(u^k,\mathcal{U}^*)$. Consequently, ${\rm dist}(x^k,\mathcal{X}^*)$ converges to $0$
		superlinearly with rate $1\!+\!\sigma$, i.e., for all $k\ge\widetilde{k}$ (if necessary by enlarging $\widetilde{k}$),
		\begin{equation}\label{super-dire}
			{\rm dist}(x^{k+1},\mathcal{X}^*) = O([{\rm dist}(x^{k},\mathcal{X}^*)]^{1+\sigma}).
		\end{equation}
		Also, by \cite[Lemma 5.3]{Ueda10} there exists a constant $c_0>0$ such that
		for all $k\ge\widetilde{k}$ (if necessary by increasing $\widetilde{k}$),
		\begin{equation}\label{dbound-xkxstar}
			\|d_{S_k}^k\|=\|d_{S_*}^k\|\le c_0{\rm dist}(u^k,\mathcal{U}^*)=c_0{\rm dist}(x^k,\mathcal{X}^*). 
		\end{equation}
		For each $k\ge\widetilde{k}$, pick any $\widetilde{x}^k\in\Pi_{\mathcal{X}^*}(x^k)$.
		By the definition of $\mathcal{X}^*$, ${\rm supp}(\widetilde{x}^k)\subseteq S_{*}$; while from
		$\lim_{k\to\infty}x^{k}=x^*$,
		we have ${\rm supp}(\widetilde{x}^k)\supseteq S_{*}$ for all $k\ge\widetilde{k}$
		(if necessary by increasing $\widetilde{k}$). Then, for all $k\ge\widetilde{k}$,
		$ {\rm supp}(\widetilde{x}^k)=S_{*}$.
		In addition, by \eqref{super-dire} there exists $\rho\in (0,1)$ such that
		${\rm dist}(x^{k+1},\mathcal{X}^*)\le\rho{\rm dist}(x^{k},\mathcal{X}^*)$ for all $k>\widetilde{k}$.
		Together with \cref{dbound-xkxstar}, for all $k\ge\widetilde{k}$ it holds that
		\begin{align*}
			\|x^k-x^{*}\|&\le\sum_{j=k}^\infty\|x^{j}\!-\!x^{j+1}\|\le\sum_{j=k}^\infty \|d^j_{S_j}\|
			\leq c_0\sum_{j=k}^\infty {\rm dist}(x^{j},\mathcal{X}^*)\\
			& < c_0\Big(\sum_{j=k}^\infty\rho^{j-k}\Big){\rm dist}(x^{k},\mathcal{X}^*)
			=\frac{c_0}{1-\rho} {\rm dist}(x^{k},\mathcal{X}^*).
		\end{align*}
		By combining this inequality and \eqref{super-dire}, it follows that for all $k>\widetilde{k}$,
		\[
		\|x^{k}-x^{*}\|\le\frac{c_0}{1-\rho}{\rm dist}(x^{k},\mathcal{X}^*)=O([{\rm dist}(x^{k-1},\mathcal{X}^*)]^{1+\sigma})
		\le O(\|x^{k-1}\!-\!x^{*}\|^{1+\sigma}).
		\]
		The desired conclusion then follows. The proof is completed.
	\end{proof}
	\begin{remark}\label{remark-superlinear}
		{\bf(a)} The local error bound condition \eqref{err-bound} is a little stronger than
		the metric subregularity of $\nabla F_{S_*}$ at $u^*$ for the origin
		because $\mathcal{U}^*$ may be a strict subset of $\nabla F_{\!S_*}^{-1}(0)$,
		but it does not require the isolatedness of $u^*$ and its local optimality. 
		
		\noindent
		{\bf (b)} The proof of the superlinear convergence of E-RNM
		in \cite{Ueda10} relies on Assumption 5.1 therein, which requires the local optimality of $x^*$. After checking its proof, we found that the local optimality of $x^*$ was only used to achieve \cite[Lemma 5.2]{Ueda10}. Thus, by following the same arguments as those for \cref{lemma-zetak}, the local optimality of $x^*$ in their Assumption 5.1 can be removed. 
	\end{remark}
	
	To conclude this section, we take a closer look at \cref{ass2}. The following lemma shows that if the regularized Newton direction $d^k$ from Step 2 satisfies condition \eqref{dboundxkxbar} for all $k\in\mathcal{K}_2$, \cref{ass2} necessarily holds. Together with  \cref{ex1} later, we conclude that \cref{ass2} is weaker than condition \eqref{dboundxkxbar} for our regularized Newton direction $d^k$.
	\begin{lemma}\label{lemma-cond}
		Suppose that \cref{ass1} holds. If $d^k$ yielded by Step 2 of \cref{hybrid} satisfies condition \eqref{dboundxkxbar} for all $k\in\mathcal{K}_2$, then \cref{ass2} holds.
	\end{lemma}
	\begin{proof}
		By \cref{well-defineness}, $\{x^k\}_{k\in\mathbb{N}}$ is bounded. Let $x^*$ be an arbitrary accumulation point of $\{x^k\}_{k\in\mathbb{N}}$.
		Then, there exists a subsequence $\{x^{k_j}\}_{j\in\mathbb{N}}$  with $k_j \in \mathcal{K}_2$
		such that $\lim_{j\rightarrow \infty} x^{k_j}=x^*$. By \cref{supp-lemma},
		for all sufficiently large $j\in\mathbb{N}$,
		${\rm supp}(x^{k_j})={\rm supp}(x^*)=S_{*}$. Write $s=|S_{*}|$.
		From the continuity, the sequence $\{G^{k_j}\}_{j\in\mathbb{N}}$ is convergent
		and let $G^*=\lim_{j\rightarrow\infty} G^{k_j}$. Clearly, $G^*$ is an $s\times s$ positive semidefinite matrix. Let $\lambda_1\ge\lambda_2\ge\cdots\ge\lambda_s\ge 0$
		be the eigenvalues of $G^*$. For each $j\in\mathbb{N}$, let $\lambda_1^{k_j}\ge\lambda_2^{k_j}\ge\cdots\ge\lambda_{s}^{k_j}>0$
		be the eigenvalues of the $s\times s$ positive definite matrix $G^{k_j}$. Then,
		for each $i\in [s]$, $\lim_{j\rightarrow\infty} \lambda^{k_j}_i =\lambda_{i}$.
		
		\noindent
		{\bf Case 1: $\lambda_s>0$}. Now the matrix $G^*$ is positive definite. Also,
		for all sufficiently large $j\in\mathbb{N}$, 
		$\lambda_s^{k_j}>\frac{\lambda_s}{2}$ and $0<\lambda_1^{k_j}\le\frac{3\lambda_1}{2}$.
		Consequently, for all sufficiently large $j\in\mathbb{N}$,
		\begin{equation}\label{suff-cond}
			\frac{-\langle r^{k_j}, d^{k_j}_{S_{k_j}}\rangle }{\| r^{k_j}\| \|d^{k_j}_{S_{k_j}}\|}
			= \frac{\langle G^{k_j} d_{S_{k_j}}^{k_j}, d^{k_j}_{S_{k_j}}\rangle }{\| G^{k_j} d_{S_{k_j}}^{k_j}\| \|d^{k_j}_{S_{k_j}}\|}
			\ge \frac{ \lambda^{k_j}_s \|d^{k_j}_{S_{k_j}}\|^2 }{\lambda_1^{k_j}\|d_{S_{k_j}}^{k_j}\|^2 }
			\ge \frac{\lambda_s}{3\lambda_1} > 0.
		\end{equation}
		
		\noindent
		{\bf Case 2: $\lambda_s=0$}. Now there exists $t\in [s]$ such that $\lambda_i=0$
		for $i=t,\ldots,s$ and $\lambda_{i}>0$ for $i=1,\ldots,t-1$. Fix any
		$0<\varepsilon<\min\left\{\frac{\varpi}{8\widehat{c}},\frac{\varpi}{4\widehat{c}\sqrt{s-t+1}}\right\}$.
		From $\lim_{j\rightarrow\infty} \lambda^{k_j}_i =\lambda_{i}$ for each $i\in[s]$ and $G^{k_j} \succ 0$ for each $j\in\mathbb{N}$,
		for all sufficiently large $j\in\mathbb{N}$,
		\begin{equation}\label{j-con}
			0<\lambda_i^{k_j}<\varepsilon\ \ {\rm for}\ i=t,\ldots,s
			\ \ {\rm and}\ \ \frac{1}{2}\lambda_{i}<\lambda_i^{k_j} <\frac{3}{2}\lambda_i\ \ {\rm for} \ i=1,\ldots,t-1.
		\end{equation}
		We claim that $t>1$. If not, $t=1$, by \cref{lemma-bound} (ii),
		$\|d^{k_j}\| = \|(G^{k_j})^{-1} r^{k_j}\| \geq \frac{\|r^{k_j}\|}{\lambda^{k_j}_1}
		\geq \frac{\varpi}{4\lambda_{1}^{k_j}}\|x^{k_j} -\overline{x}^{k_j}\|
		\ge \frac{\varpi}{4\varepsilon}\|x^{k_j} -\overline{x}^{k_j}\|$,
		which along with $\varepsilon\le\frac{\varpi}{8\widehat{c}}$ implies that
		$\|d^{k_j}\|>2\widehat{c}\|x^{k_j} -\overline{x}^{k_j}\|$,
		a contradiction to condition \eqref{dboundxkxbar}.
		Now let $G^{k_j}$ have the eigenvalue decomposition
		given by $G^{k_j}=(V^{k_j})^{\mathbb{T}}{\rm diag}(\lambda_1^{k_j},\ldots,\lambda_s^{k_j})V^{k_j}$,
		where $V^{k_j}$ is an $s\times s$ orthogonal matrix. For each $j\in\mathbb{N}$,
		since the column vectors $v_1^{k_j},\ldots,v_s^{k_j}$ of the matrix $V^{k_j}$
		are linearly independent, there exist $\gamma_1^{k_j},\ldots,\gamma_s^{k_j}\in\mathbb{R}$
		such that
		\begin{equation}\label{eq1}
			\frac{r^{k_j}}{\|r^{k_j}\|}=\sum_{i=1}^s \gamma_i^{k_j} v^{k_j}_i\ \ {\rm with}\ \ \sum_{i=1}^s (\gamma_i^{k_j})^2 = 1.
		\end{equation}
		Together with the definition of $d^{k_j}_{S_{k_j}}$, it follows that
		\begin{equation}\label{eq2}
			\frac{d^{k_j}_{S_{k_j}}}{\|r^{k_j}\|}
			=\frac{-(G^{k_j})^{-1}r^{k_j}}{\|r^{k_j}\|}
			=-\sum_{i=1}^s \frac{\gamma_i^{k_j}}{\lambda_i^{k_j}}v_i^{k_j}.
		\end{equation}
		By combining condition \eqref{dboundxkxbar} and \cref{lemma-bound} (ii),
		for each $j\in\mathbb{N}$, we have
		\begin{equation}\label{eq3}
			\|d^{k_j}_{S_{k_j}}\|=\|d^{k_j}\|\le(4{\widehat{c}}/{\varpi})\|r^{k_j}\|,
		\end{equation}
		which by \eqref{eq2} means that
		$\sum_{i=1}^s\big({\gamma_i^{k_j}}/{\lambda_i^{k_j}}\big)^2 \leq\frac{16\widehat{c}^2}{\varpi^2}$. This by \eqref{j-con} implies that
		for all sufficiently large $j\in\mathbb{N}$,  $\gamma_i^{k_j}\le\frac{4\varepsilon\widehat{c}}{\varpi}$
		with $i\in\{t,\ldots,s\}$. Together with $\sum_{i=1}^s (\gamma_i^{k_j})^2 = 1$,
		we obtain that $\sum_{i=1}^{t-1}(\gamma_i^{k_j})^2\ge 1-\frac{16(s-t+1)\varepsilon^2\widehat{c}^2}{\varpi^2}$
		and then for all sufficiently large $j\in\mathbb{N}$, there exists $l_j\in\{1,\ldots,t\!-\!1\}$ such that
		$(\gamma_{l_j}^{k_j})^2\geq \frac{\varpi^2-16(s-t+1)\varepsilon^2\widehat{c}^2}{\varpi^2(t-1)}$.
		Thus, for all sufficiently large $j\in\mathbb{N}$, 
		it follows from \cref{eq1}-\eqref{eq3} that
		\[
		\frac{-\langle r^{k_j}, d^{k_j}_{S_{k_j}}\rangle}{\|r^{k_j}\| \|d^{k_j}_{S_{k_j}}\|}
		\ge\frac{\varpi}{4\widehat{c}}\sum_{i=1}^s \frac{(\gamma_i^{k_j})^2}{\lambda_{i}^{k_j}}
		\ge\frac{\varpi(\gamma_{l_j}^{k_j})^2}{4\widehat{c}\lambda_{l_j}^{k_j}}
		\ge\frac{\varpi^2-16(s-t+1)\varepsilon^2\widehat{c}^2}{6\widehat{c}\lambda_{l_j}\varpi(t-1)}>0,
		\]
		where the third inequality is also using $\lambda_l^{k_j}\leq \frac{3}{2}\lambda_l$ by \eqref{j-con},
		and the last one is by $0<\varepsilon<\frac{\varpi}{4\widehat{c}\sqrt{s-t+1}}$.
		From the last inequality and \eqref{suff-cond}, we obtain the conclusion.
	\end{proof}
	
	The example below shows that the inverse of \cref{lemma-cond} does not hold.
	
	\begin{example}\label{ex1}
		Consider the problem $\min_{t\in\mathbb{R}} f(t) + |t|^{\frac{1}{2}}$
		with $f$ defined as follows:
		\[
		f(t)\!:=\left\{\begin{array}{cl}
			\frac{49}{8}t^2 - \frac{67}{4}t + \frac{85}{8} &{\rm if}\ t \in (-\infty, 1),\\
			(t-2)^4 - t^{\frac{1}{2}} & {\rm if}\ t\in [1,4),\\
			\frac{1537}{64}t^2-\frac{5132}{32}t + \frac{1085}{4} &{\rm if}\ t \in [4, \infty).
		\end{array}\right.
		\]
		We use \cref{hybrid} with $\widetilde{\tau}=2$, $\widetilde{\alpha} = 1, \mu_{\min}=40,\widetilde{L}=49$ and $\sigma=\frac{1}{3}, b_2=1,\varrho=10^{-4},\beta=\frac{1}{2},t^0=2.1$
		to seek a critical point of this problem.
		From the iterates of \cref{hybrid},
		the generated sequence $\{t^k\}$ satisfies $\lim_{t\rightarrow \infty} t^k = 2$.
		When $t^k$ is sufficiently close to $2$, all the iterates are from regularized Newton step and
		$|d^k| =\left|\frac{4(t^k - 2)^3}{12(t^k-2)^2 + 4^{\frac{1}{3}} (t^k -2)}\right|= O(|t^k-2|^2)$,
		while by \cref{lemma-bound} (ii) and \cref{lemma-gap} we have
		$|t_k - \overline{t}_k| = O(|f'(t_k)|) = O(|t_k - 2|^3)$.
		Then, $|t^k - \overline{t}^k| = o(|d^k|)$ and the condition in \eqref{dboundxkxbar}
		does not hold for all sufficiently large $k$. However, \cref{ass2} always holds
		because $-\frac{d^k f'(t^k)}{|d^k| |f'(t^k)|} = 1$ for all $k$.
	\end{example}
	\section{Numerical experiments}\label{sec5}
	
	In this section we apply HpgSRN to solving the $\ell_q$-norm regularized linear
	and logistic regression problems on real data, which respectively take the form of \eqref{model} with $f=f_1$ or $f_2$, where $f_1(z)\!:=\frac{1}{2}\|z-b\|^2$ and
	$f_2(z)\!:=\sum_{i=1}^m\log\big(1\!+\!\exp(-b_iz_i)\big)$ for $z\in\mathbb{R}^m$.
	Here, $b\in\mathbb{R}^m$ is a given vector. Clearly, such $f$ satisfies \cref{ass1}
	and the associated $F$ is a KL function. All numerical tests
	are conducted on a desktop running in MATLAB R2020b and 64-bit Windows
	System with an Intel(R) Core(TM) i7-10700 CPU 2.90GHz and 32.0 GB RAM.  The MATLAB code is available at \url{https://github.com/YuqiaWU/HpgSRN}.
	\subsection{Implementation of HpgSRN}\label{sec5.1}
	
	In \cref{hybrid}, we set $\mu_0 = 1$ and when $k\geq 1$, $\mu_k$ is chosen
	by the Barzilai-Borwein (BB) rule \cite{BBstep}, that is,
	\[
	\mu_k=\max\Big\{\mu_{\min}, \min\Big\{\mu_{\max},\frac{\langle x^k\!-\!x^{k-1},
		\nabla\psi(x^k)\!-\!\nabla\psi(x^{k-1})\rangle}{\|x^k\!-\!x^{k-1}\|^2}\Big\}\Big\}
	\]
	with $\mu_{\min}\!=\!10^{-20}$ and $\mu_{\max}\!=\!10^{20}$.
	For each $k\in\mathcal{K}_2$, we call the MATLAB function \textsf{eigs} to compute the approximate smallest eigenvalue of $\nabla^2 F_{S_k}(u^k)$, which requires about $O(|S_k|^2)$ flops by \cite{stewart02}.
	Since $|S_k|$ is usually much smaller than $n$, this computation cost is not expensive.
	In addition, we choose
	\[
	\widetilde{\tau}=10,\,\widetilde{\alpha}=10^{-8},\,
	\sigma = 0.5,\, b_1 = 1+10^{-8},\, b_2 = 10^{-3}, \varrho = 10^{-4}, \beta = 2.
	\]
	During the testing, we solve the linear system in (S4) via a direct method
	if $|S_k|<500$, otherwise a conjugate gradient method.   The direct method for computing the inverse of the  $G^k$ needs about $O(|S_k|^3)$ flops, so that HpgSRN is well adapted to high dimensional problems if $|S_k|$ is small.
	Our preliminary tests indicate that \eqref{model} with $q=1/2$ usually
	has better performance than \eqref{model} with other $q\in(0,1)$
	in terms of the CPU time and the sparsity. This coincides with the conclusion
	in \cite{Hu17,xu10}. Inspired by this, we choose $q=1/2$ for the subsequent
	numerical testing. The parameter $\lambda$ in \eqref{model}
	is specified in the corresponding experiments.
	
	We compare the performance of HpgSRN with that of ZeroFPR \cite{Themelis18}.
	The code package of ZeroFPR is downloaded from \url{http://github.com/kul-forbes/ForBES}.
	Consider that the iterate steps of  PG method with a monotone line search (PGls), a monotone version of SpaRSA \cite{Wright09}, are the same as those of Algorithm \ref{gist} with the above BB rule for updating $\mu_k$. We also compare the performance of HpgSRN with that of PGls to check the effect of the additional subspace regularized Newton step on HpgSRN. The parameters of  PGls
	are chosen to be the same as those involved in Step 1 of HpgSRN except $\widetilde{\tau}=2$.  For the three algorithms, we adopt the same
	stopping criterion  $\gamma\|x^k - {\rm prox}_{\gamma^{-1}(\lambda g)}(x^k - \gamma^{-1}\nabla \!\psi(x^k))\|_{\infty} < 10^{-3}$  or $k\ge 50000$, where $\gamma = L/0.95$ and $L$ is  an estimation of the Lipschitz constant of $\nabla\psi(\cdot)$. It is well known that the Lipschitz constants of $A^{\mathbb{T}}\nabla f_1(A\cdot)$ and $A^{\mathbb{T}}\nabla f_2(A\cdot)$ are $\|A\|^2$ and $0.25\|A\|^2$, respectively. We use the following MATLAB  code to estimate the spectral norm of $A$:
	
	\textsf{Amap = @(x) A*x; ATmap = @(x) A'*x; AATmap = @(x) Amap(ATmap(x));}
	
	\textsf{eigsopt.issym = 1; L = eigs(AATmap, m, 1, 'LA', eigsopt).}
	
	\noindent
	As in ZeroFPR, we choose $x^0=0$ as the starting point. Although $x^0 = 0$ is a local minimizer of $F$ and hence an $L$-type stationary point by \cite[Theorem 4.4]{themelis21}, it is not a good one in terms of objective value; see the difference between $F(0)$ and Fval, the objective value of the output, for each example in Tables \ref{tab1} and \ref{tab2}. It is worth noting that equipped with such an initial point, Algorithm \ref{hybrid} may stop in the first iteration and in this case, $x^0$ is regarded as an acceptable solution.
	
	
	
	In the next two subsections, we will conduct the experiments on real data and report the numerical results including the number of iterations (Iter\#), the CPU times in seconds (Time), the objective function values (Fval) and the cardinality of the outputs (Nnz).
	In particular, to check the effect of the regularized Newton steps in HpgSRN,
	we record its number of iterations in the form $M(N)$,
	where $M$ means the total number of iterates and $N$ means the number of regularized Newton steps.
	\subsection{\texorpdfstring{$\ell_q$}--norm regularized linear regression}\label{sec5.3}
	
	We conduct the experiments for the $\ell_q$-norm regularized linear regressions
	with $(A,b)$ from LIBSVM datasets (see \url{https://www.csie.ntu.edu.tw}).
	As suggested in \cite{huang10}, for {\bf housing} and {\bf space$\_$ga},
	we expand their original features with polynomial basis functions.
	The second column of Table \ref{tab1} lists the values of $\|A\|^2$ and $F(0)$,
	which reflect the condition number of the Hessian matrix of
	the loss function $\psi$ and the quality of the starting point $x^0$  respectively.
	For each dataset, we solve \eqref{model} associated to $f_1$ and
	$\lambda=\lambda_c \|A^{\mathbb{T}}b\|_{\infty}$ for two different $\lambda_c$'s with the three solvers.
	
	From \cref{tab1}, we see that for all test examples HpgSRN spends much less time than ZeroFPR and PGls. For example, for {\bf log1p.E2006.train} with $\lambda_c = 10^{-5}$, ZeroFPR and PGls
	require more than one hour to yield an output,
	but HpgSRN returns an output within only $314s$.
	In terms of the objective function value and sparsity, the outputs of HpgSRN are comparable with those of ZeroFPR and PGls,
	and even in some examples, these outputs of HpgSRN are better.
	For example, for {\bf housing7} with both $\lambda_c$'s the  objective function values of HpgSRN are better than those of ZeroFPR and PGls as well as the sparsity of HpgSRN is much less.
	
	\begin{table}[!ht]
		\centering
		\tiny
		\caption{Numerical comparisons on $\ell_q$-norm regularized linear regressions with LIBSVM datasets\label{tab1}}
		\begin{tabular}{lcccccc}
			\toprule
			\begin{tabular}[l]{@{}l@{}}Data\\ $(m,n)$ \end{tabular} & \begin{tabular}[l]{@{}l@{}}$\|A\|^2$ \\ $F(0)$ \end{tabular} & $\lambda_c$ & Index & HpgSRN & ZeroFPR & PGls   \\
			\midrule
			\multirow{8}{*}{\begin{tabular}[l]{@{}l@{}}space$\_$ga9\\ $(3107,5505)$ \end{tabular}} & \multirow{8}{*}{\begin{tabular}[l]{@{}l@{}}4.01e3\\ 5.77e3 \end{tabular}} & \multirow{4}{*}{$10^{-3}$}  & Iter\# & 17(5)  & 43  & 180       \\
			& & & Time & 0.45  & 0.98 &  0.93     \\
			&  & & Fval & 36.47 & 37.24  & 37.15  \\
			&  & & Nnz & 7  & 7 & 6 \\
			
			\cline{3-7}
			
			& & \multirow{4}{*}{$10^{-4}$}       & Iter\# & 230(64)  & 476  & 3058    \\
			&& & Time &  2.26 & 9.03 &  16.48     \\
			&& & Fval & 20.93  & 20.31  & 21.57  \\
			&& & Nnz & 15 & 19  & 15\\
			
			\midrule
			\multirow{8}{*}{\begin{tabular}[l]{@{}l@{}}housing7\\ $(506,77520)$\end{tabular}} &  \multirow{8}{*}{\begin{tabular}[l]{@{}l@{}}3.28e5\\ 1.50e5 \end{tabular}} &
			\multirow{4}{*}{$10^{-3}$}   & Iter\# &  639(157) & 4164  & 25133      \\
			& & & Time &  14.45 & 2.13e2  & 4.08e2     \\
			& &  & Fval & 2.25e3  & 2.57e3  & 2.56e3     \\
			&  & & Nnz & 27  & 49  & 57    \\
			
			\cline{3-7}
			
			& 	& \multirow{4}{*}{$10^{-4}$}      & Iter\# &  1765(485) & 18807 & 50000       \\
			& && Time & 49.26 & 9.81e2 & 8.59e2     \\
			&& & Fval & 8.89e2 & 9.27e2  & 9.17e2    \\
			&& & Nnz & 82  & 123  & 135   \\
			
			\midrule
			\multirow{8}{*}{\begin{tabular}[l]{@{}l@{}}E2006.test\\ $(3308,72812)$\end{tabular}} & \multirow{8}{*}{\begin{tabular}[l]{@{}l@{}}4.79e4\\ 2.46e4 \end{tabular}} & \multirow{4}{*}{$10^{-4}$}   & Iter\# & 3(0)  &  3 & 3      \\
			&&  & Time & 0.03  &  0.25 & 0.03    \\
			&&   & Fval & 2.45e2 & 2.45e2  & 2.45e2     \\
			&&   & Nnz & 1 & 1  & 1    \\
			
			\cline{3-7}
			
			&& \multirow{4}{*}{$10^{-5}$}      & Iter\# & 3(0) & 4  & 4      \\
			&& & Time & 0.05 & 0.25 &  0.04     \\
			&& & Fval & 2.40e2  & 2.40e2  & 2.40e2 \\
			&& & Nnz & 1  & 1  & 1  \\
			
			\midrule
			\multirow{8}{*}{\begin{tabular}[l]{@{}l@{}}E2006.train\\ $(16087,150348)$\end{tabular}}
			& \multirow{8}{*}{\begin{tabular}[l]{@{}l@{}} 1.91e5\\ 1.03e5 \end{tabular}}
			& \multirow{4}{*}{$10^{-4}$}      & Iter\# & 3(0)  & 3  & 3      \\
			&& & Time & 0.09 & 1.06 & 0.09       \\
			&& & Fval & 1.22e3  &  1.22e3 & 1.22e3   \\
			&& & Nnz & 1  &  1 & 1  \\
			
			\cline{3-7}
			
			&& \multirow{4}{*}{$10^{-5}$}      & Iter\# & 4(0) & 4  & 4   \\
			&& & Time & 0.11 & 1.05 & 0.11    \\
			&& & Fval & 1.20e3  & 1.20e3  & 1.20e3    \\
			&& & Nnz & 1  &  1 & 1  \\
			
			\midrule
			\multirow{8}{*}{\begin{tabular}[l]{@{}l@{}}log1p.E2006.test\\ $(3308,1771946)$\end{tabular}} & \multirow{8}{*}{\begin{tabular}[l]{@{}l@{}} 1.46e7\\2.46e4 \end{tabular}} & \multirow{4}{*}{$10^{-4}$}      & Iter\# &  372(88) & 827  &  1416   \\
			&	& & Time & 33.54 & 2.87e2 & 1.16e2     \\
			&	& & Fval & 2.35e2  & 2.43e2  & 2.37e2   \\
			&	& & Nnz &  5 & 4  & 6   \\
			
			\cline{3-7}
			
			&	& \multirow{4}{*}{$10^{-5}$}      & Iter\# &  755(166) & 6708  & 22305      \\
			&	& & Time & 1.01e2 & 2.28e3 & 2.30e3      \\
			&	& & Fval & 1.54e2  & 1.53e2  & 1.49e2  \\
			&	& & Nnz &  385 & 460 & 389    \\
			
			\midrule
			\multirow{8}{*}{\begin{tabular}[l]{@{}l@{}}log1p.E2006.train\\ $(16087,4265669)$\end{tabular}} & \multirow{8}{*}{\begin{tabular}[l]{@{}l@{}} 5.86e7\\ 1.03e5 \end{tabular}} & \multirow{4}{*}{$10^{-4}$}      & Iter\# & 286(58)  & 855  & 1621       \\
			&	& & Time & 77.95 & 8.57e2 &  3.85e2      \\
			&	& & Fval & 1.16e3  & 1.16e3  & 1.16e3    \\
			&	& & Nnz &  7 & 5  & 4  \\
			
			\cline{3-7}
			
			&	& \multirow{4}{*}{$10^{-5}$}      & Iter\# & 944(195) & 5610 & 33112   \\
			&	& & Time & 3.14e2 & 5.26e3 &  8.83e3    \\
			&	& & Fval &  1.02e3 & 1.02e3 & 1.01e3   \\
			&	& & Nnz &  141 & 184  & 155  \\
			\bottomrule
		\end{tabular}
	\end{table}
	\subsection{\texorpdfstring{$\ell_q$}--norm regularized logistic regression}\label{sec5.4}
	
	We conduct the experiments for the $\ell_q$-norm regularized logistic regressions
	with $(A,b)$ from LIBSVM datasets. For each data, we solve \eqref{model}
	associated to $f_2$ and $\lambda=\lambda_c \max_{1\le j\le n}\|A_j\|_1$
	for two different $\lambda_c$'s with the three solvers.
	\cref{tab2} records their numerical results.
	We see that in terms of CPU time, HpgSRN is still the best one among the three solvers;
	in terms of the quality of the other outputs, HpgSRN has a comparable performance with ZeroFPR and PGls.
	
	\begin{table}[!ht]
		\centering
		\tiny
		\caption{Numerical comparisons on $\ell_q$-norm regularized logistic regressions with LIBSVM datasets\label{tab2}}
		\begin{tabular}{lcccccc}
			\toprule
			\begin{tabular}[l]{@{}l@{}}Data\\ $(m,n)$ \end{tabular} & \begin{tabular}[l]{@{}l@{}}$\|A\|^2$ \\ $F(0)$ \end{tabular} & $\lambda_c$& Index & HpgSRN & ZeroFPR &  PGls \\
			\midrule
			\multirow{8}{*}{\begin{tabular}[l]{@{}l@{}} colon-cancer\\ $(62,2000)$\end{tabular}} & \multirow{8}{*}{\begin{tabular}[l]{@{}l@{}}1.94e4\\ 42.98 \end{tabular}} & \multirow{4}{*}{$10^{-2}$}   & Iter\# &  48(6) & 730  & 94    \\
			&   &  & Time & 0.04 & 0.74  &  0.06     \\
			&  &   & Fval & 7.97 & 10.58 & 7.77  \\
			&  &   & Nnz & 10  & 9 & 9  \\
			
			\cline{3-7}
			
			&  & \multirow{4}{*}{$ 10^{-3}$}      & Iter\# & 94(9) & 1853  & 175     \\
			& & & Time &  0.07 & 2.07 &  0.11   \\
			& & & Fval & 1.03  & 1.07 & 1.07  \\
			& & & Nnz & 11 & 12  & 12 \\
			
			\midrule
			\multirow{8}{*}{\begin{tabular}[l]{@{}l@{}}rcv1\\ $(20242,47236)$\end{tabular}} &  \multirow{8}{*}{\begin{tabular}[l]{@{}l@{}}4.48e2\\ 1.40e4 \end{tabular}} &\multirow{4}{*}{$ 10^{-2}$}   & Iter\# & 65(10)  & 448  & 1193  \\
			&  &  & Time & 1.00  & 6.35  &  11.24    \\
			& &   & Fval & 4.23e3 & 4.35e3 & 4.24e3     \\
			& &   & Nnz & 165  & 167 & 164  \\
			
			\cline{3-7}
			
			&  & \multirow{4}{*}{$10^{-3}$}      & Iter\# & 365(96) & 2081 & 5536 \\
			&  & & Time & 7.78 & 29.27 & 88.65   \\
			&  & & Fval & 1.28e3 & 1.53e3  & 1.27e3   \\
			&  & & Nnz & 704 & 741  & 717 \\
			
			\midrule
			\multirow{8}{*}{\begin{tabular}[l]{@{}l@{}}news20 \\ $(19996,1355191)$\end{tabular}} & \multirow{8}{*}{\begin{tabular}[l]{@{}l@{}}1.73e3\\ 1.39e4 \end{tabular}} & \multirow{4}{*}{$10^{-2}$}   & Iter\# & 44(6)  & 170 & 981 \\
			&   &  & Time & 2.65  &  36.61 & 53.14  \\
			&  &   & Fval & 9.73e3 & 1.04e4 & 9.53e3  \\
			&  &   & Nnz & 51 & 42  & 50    \\
			\cline{3-7}
			
			&  & \multirow{4}{*}{$10^{-3}$}      & Iter\# &  410(99) & 1528 & 18538  \\
			& & & Time & 41.45 & 3.44e2 & 1.43e3     \\
			& & & Fval & 4.31e3 & 4.71e3   & 4.25e3   \\
			& & & Nnz & 385  & 371 & 401   \\
			
			\bottomrule
		\end{tabular}
	\end{table}
	
	\medskip
	To sum up, HpgSRN requires the least CPU time for all the test examples compared to ZeroFPR and PGls,
	and for those large scale examples, HpgSRN is at least ten times faster than ZeroFPR and PGls.
	The outputs of
	the objective function value and the sparsity yielded by HpgSRN have a comparable even better quality. This indicates that the introduction of second-order steps
	improves greatly the performance of the first-order method. We also observe that for most of examples, the iterates generated by the regularized Newton step account for about $10\%$--$35\%$ of the total iterates.
	
	\section{Conclusion} \label{sec6}
	
	For the $\ell_q$-norm regularized composite problem \eqref{model}, we proposed
	a globally convergent regularized Newton method by exploiting the special structures
	of the $\ell_q$-norm. This is another attempt to combine a first-order method and a second-order method while maintaining the good properties of both methods.
	We not only established the convergence of the whole iterate sequence
	under a mild curve-ratio condition and the KL property of $F$,
	but also achieved a superlinear convergence rate under an additional local error bound
	condition. In particular, the local superlinear convergence result neither requires
	the isolatedness of the limit point nor its local minimum property.

	\section*{Acknowledgments} The authors are grateful to the associate editor and two reviewers for their valuable suggestions and remarks which allowed them to improve the original presentation of the paper.

	\bibliographystyle{siamplain}
	\bibliography{references}
\end{document}